\def\ARXIV{1}
\def\ACC{2}
\def\MODE{\ARXIV}

\documentclass[letterpaper, 10 pt, conference]{ieeeconf}  

\IEEEoverridecommandlockouts                              

\usepackage{graphicx} 
\usepackage{epsfig} 
\usepackage{mathptmx} 
\usepackage{algorithm}
\usepackage{algpseudocode}
\usepackage{times} 
\usepackage{amsmath} 
\usepackage{amssymb}  
\usepackage{xcolor}
\usepackage{subfigure}
\usepackage{lipsum}

\newtheorem{theorem}{Theorem}
\newtheorem{lemma}{Lemma}
\newtheorem{problem}{Problem}
\newtheorem{remark}{Remark}

\newcommand{\norm}[1]{\left\lVert#1\right\rVert}
\newcommand{\qApprox}{\textrm{quad}}

\newcommand{\R}{\mathbb{R}}
\newcommand{\J}{\mathcal{J}}

\newcommand{\sigmoid}{\text{sigmoid}}

\title{\LARGE \bf
  Differential Dynamic Programming for Nonlinear Dynamic Games
}

\author{Bolei Di$^{1}$ and Andrew Lamperski$^{2}$
  \thanks{The authors are with the Department of Electrical and
    Computer Engineering, University of Minnesota, Minnesota, USA}
  \thanks{$^{1}$ {\tt\small dixxx047@umn.edu}, $^{2}$ {\tt\small alampers@umn.edu}}%
}

\begin{document}

\maketitle
\thispagestyle{empty}
\pagestyle{empty}

\begin{abstract}
  Dynamic games arise when multiple agents with differing objectives
  choose control inputs to a dynamic system.
  Dynamic games model a wide variety of applications in economics,
  defense, and energy systems.
  However, compared to single-agent control problems, the
  computational methods for dynamic games are relatively limited.
  As in the single-agent case, only very specialized dynamic games can
  be solved exactly, and so approximation algorithms are
  required.
  This paper extends the differential dynamic programming algorithm
  from single-agent control to the case of non-zero sum
  full-information dynamic games. 
  The method works by computing quadratic approximations to the
  dynamic programming equations.
  The approximation results in static quadratic games which are solved
  recursively.
  Convergence is proved by showing that the algorithm
  iterates sufficiently close to iterates of Newton's method to
  inherit its convergence properties.
  A numerical example is provided. 

\end{abstract}


\section{INTRODUCTION}

Dynamic games arise when multiple agents with differing objectives act
upon a dynamic system.
In contrast, optimal control can be viewed as the specialization of
dynamic games to the case of a single agent.
Dynamic games have many applications including
 pursuit-evasion \cite{rusnak2005lady},
active-defense
\cite{prokopov2013linear,garcia2014cooperative},
economics \cite{el2013dynamic} and the smart grid
\cite{zhu2012differential}.
Despite a wide array of applications, the computational methods for
dynamic games are considerably less developed than the single-agent
case of optimal control.

This paper shows how the differential dynamic programming (DDP) method from
optimal control \cite{jacobson1968new} extends to discrete-time
non-zero sum dynamic games.
Closely related works from \cite{sun2015game,sun2016stochastic} focus
on the case of zero-sum dynamic games.
Classical differential dynamic programming operates by iteratively solving
quadratic approximations to the Bellman equation from optimal control.
Our method applies similar methods to the generalization of the Bellman
equation for dynamic games \cite{basar1999dynamic}.
Here, at each stage, the algorithm solves a static game formed by
taking quadratic approximations to the value function of each agent.
We show that the algorithm converges quadratically in the
neighborhood of a strict Nash equilibrium.
To prove convergence, we extend arguments from
\cite{murray1984differential}, which relate DDP iterates to those of
Newton's method, to the case of dynamic games.
In particular, we extend the recursive solution for Newton's method
\cite{dunn1989efficient} to dynamic games, and demonstrate that the
solutions produced Newton's method and the DDP method are close.

\subsection{Related Work}

A great deal of work on algorithmic solutions to dynamic
games has been done.
This subsection reviews related work which is a bit more removed from
the closer references described above.
As we will see, most works solve somewhat different problems compared
to the current paper.

Methods for finding Nash equilibria via
extremum seeking were presented in
\cite{frihauf2012nash,frihauf2013finite,frihauf2013nash}. In
particular, the controllers drive the states of a dynamic system to
Nash equilibria of static games.
A related
method for linear quadratic games was presented in
\cite{pan2004sliding}.
For these works, each agent only requires measurements of its own
cost. However, it is limited to finding equilibria in steady
state. Our method requires each agent to have explicit model
information, but gives equilibria over finite horizons. This is
particularly important for games in which trajectories from initial to
final states are desired.


Several works focus on the solution to dynamic potential games.
Potential games are more tractable than general dynamic games, as they
can be solved using methods from single-agent optimal control
\cite{gonzalez2016survey,gonzalez2013discrete,gonzalez2014dynamic,mazalov2017linear,zazo2015new,zazo2015dynamic}. 
However, potential games satisfy restrictive symmetry
conditions.
In particular, the assumption precludes interesting applications with
heterogeneous agents. 

\subsection{Paper Outline}
The general problem is formulated in Section~\ref{sec:problem}. The
algorithm is described in Section~\ref{sec:alg} and the convergence
proof is sketched in Section~\ref{sec:convergence}. A numerical
example is described in~\ref{sec:example}. 
\if\MODE\ARXIV{
Conclusions and future
directions are discussed in~\ref{sec:conclusion} while
  the proof details are given in the appendix.
}
\fi
\if\MODE\ACC{
  Conclusions and future
  directions are discussed in~\ref{sec:conclusion} while
  some of the
  proof details are given in the appendix. However, due to space
  constraints, some details are omitted.
}
\fi


\section{DETERMINISTIC NONLINEAR DYNAMIC GAME PROBLEM}
\label{sec:problem}
In this section, we introduce deterministic finite-horizon nonlinear game problem, the notations for the paper, the solution concept and convergence criterion of our proposed method.

The main problem of interest is a deterministic full-information dynamic game of the
form below.

\begin{problem}
  \label{problem:original}
  {\it
    \textbf{Nonlinear dynamic game} \\
    Each player tries to minimize their own cost
    \begin{align}
     \label{eq:cost}
      J_n (u) = \sum_{k = 0}^T c_{n, k}(x_k, u_{:,k})
    \end{align}
    Subject to constraints
    \begin{subequations}
      \label{eq:dynamics}
      \begin{align}
        & x_{k+1} = f_k(x_k, u_{:,k}) \\
        & x_0 \textrm{ is fixed.}
      \end{align}
    \end{subequations}
  }
\end{problem}
Here, the state of the system at time $k$ is denoted by $x_k\in
\mathbb{R}^{n_x}$.

Player $n$'s
input at time $k$ is given by $u_{n, k} \in
\mathbb{R}^{n_{u_{n}}}$. The vector of all player actions at time $k$ is
denoted by: $u_{:, k} = [u_{1,k}^\top, u_{2,k}^\top, \ldots,
u_{N,k}^\top]^\top \in \mathbb{R}^{n_u}$.
The cost for player $n$ at time $k$ is $c_{n,k}(x_k,u_{:,k})$. This
encodes the fact that the costs for each player can depend on the
actions of all the players.

In later analysis, some other notation will be helpful.
The vector player $n$'s
actions over all time is denoted by $u_{n,:} = [u_{n,0}^\top, u_{n,1}^\top, \ldots,
u_{n,T}^\top]^\top$. The vector of all actions other than those of
player $n$ is denoted by
$u_{-n,:} = [u_{1,:}^\top,\ldots,u_{n-1,:}^\top,u_{n+1,:}^\top,\ldots,u_{N,:}^\top]^\top$.
The vector of all states is denoted by
$x=[x_0^\top,x_1^\top,\ldots,x_T^\top]^\top$ while the vector of all inputs
is given by $u = [u_{1,:}^\top, u_{N,:}^\top, \ldots,
u_{N,:}^\top]^\top$.

Note that since the initial state is fixed and the dynamics are
deterministic, the costs for each player can be expressed as functions
of the vector of actions, $J_n(u)$.
%


A \emph{local Nash equilibrium} for problem
\ref{problem:original} is a set of inputs $u^{\star}$ such that
\begin{equation}
  \label{eq:LNE}
  J_n(u_{n,:}, u_{-n, :}^{\star}) \geq J_n(u^{\star}), \  n = 1, 2, \ldots, N
\end{equation}
for all $u_{n,:}$ in a neighborhood of $u_{n,:}^\star$. In the context
of dynamic games, this correponds to an open-loop, local Nash
equilibrium \cite{basar1999dynamic}.
The equilibrium is called a \emph{strict local Nash equilibrium} if
the inequality in \eqref{eq:LNE} is strict for all $u_{n,:}\ne
u_{n,:}^\star$ in a neighborhood of $u_{n,:}^\star$.


In this paper, we focus on computing Nash equilibria by solving the
following necessary conditions for local Nash equilibria:
\begin{problem}
  {\it
    \textbf{Necessary conditions}
    \begin{equation}
      \label{problem:necessary}
      \frac{\partial J_n}{\partial u_{n,:}} = 0
    \end{equation}
    for $n=1,\ldots,N$.
  }
\end{problem}

For convenient notation, we stack all of the gradient vectors from
\eqref{problem:necessary} into a single vector:
\begin{align}
  \label{eq:JFun}
  \mathcal{J}(u) =
  \begin{bmatrix}
    \frac{\partial J_1}{\partial u_{1,:}} & \frac{\partial J_2}{\partial u_{2,:}} & \cdots & \frac{\partial J_N}{\partial u_{N,:}}
  \end{bmatrix}^\top.
\end{align}
Thus, the necessary condition is equivalent to $\J(u) = 0$. Such
conditions arise in works such as \cite{facchinei2007generalized,dutang2013survey}.

We will present a method for solving these necessary conditions for a
local Nash equilibrium via differential dynamic programming (DDP). In
principle, an input vector satisfying the necessary conditions
$\J(u)=0$ could be found via Newton's method. Similar to the
single-player case from \cite{murray1984differential}, we analyze the
convergence properties of DDP by proving that its solution is close to
that computed by Newton's method.

To guarantee convergence, we assume that $\J(u)$ satisfies the
smoothness and non-degeneracy conditions required by Newton's method \cite{nocedal2006numerical}.
For smoothness,  we assume that $\mathcal{J}(u)$ is
differentiable with locally Lipschitz derivatives.
For non-degeneracy, we assume that $\frac{\partial
  \mathcal{J}(u^*)}{\partial u}$ is invertible. A sufficient condition
for the smoothness assumptions is that the functions $f_k$ and
$c_{n,k}$ are twice continuously differentiable with Lipschitz second
derivatives. In our DDP solution, we will solve a sequence of stage-wise
quadratic games. As we will see, a sufficient condition for invertibility of
$\frac{\partial \J(u^*)}{\partial u}$ is the unique solvability of the
stage-wise games near the equilibrium.

\section{Differential Dynamic Programming Algorithm}
\label{sec:alg}

This section describes the differential dynamic programming algorithm
for dynamic games of the form in Problem~\ref{problem:original}.
Subsection~\ref{sec:algOverview} gives a high-level description of the
algorithm, while Subsection~\ref{sec:algMatrices} describes the
explicit matrix calculations used in algorithm. 

\subsection{Algorithm Overview}
\label{sec:algOverview}
The equilibrium solution to the general dynamic game can be characterized by the
Bellman recursion:
\begin{subequations}
  \begin{align}
    \label{eq:bellman}
    V_{n,T+1}^\star(x_{T+1}) &= 0 \\
    Q_{n,k}^\star(x_k,u_{:,k}) &= c_{n,k}(x_k,u_{:,k}) +
                                 V_{n,k+1}^\star(f_k(x_k,u_{:,k})) \\
    \label{eq:QGame}
    V_{n,k}^\star(x_k) &= \min_{u_{n,k}} Q_{n,k}^\star(x_k,u_{:,k}).
  \end{align}
\end{subequations}
In particular, if a solution to the Bellman recursion is found, the
corresponding optimal strategy for player $n$ at time $k$ would be the
$u_{n,k}$ which minimizes $Q_{n,k}^\star(x_k,u_{:,k})$. Note that \eqref{eq:QGame} defines a static game with respect to the $u_{:,k}$ variable at step $k$.

The idea of the differentiable dynamic programming (DDP) is to
maintain quadratic approximations of $V_{n,k}^*$ and $Q_{n,k}^*$
denoted by $\tilde V_{n,k}$ and $\tilde Q_{n,k}$, respectively.

We need some notation for our approximations.
For a scalar-valued function, $h(z)$, we denote the quadratic
approximation near $\bar{z}$ by:
\begin{subequations}
  \label{eq:qApprox}
  \begin{align}
    \qApprox(h(z))_{\bar z} =&\frac{1}{2}
                               \begin{bmatrix}
                                 1 \\
                                 \delta z
                               \end{bmatrix}^\top
    \begin{bmatrix}
      2 h(\bar z) & \frac{\partial h(\bar z)}{\partial z} \\
      \frac{\partial h(\bar z)}{\partial z}^\top & \frac{\partial^2
        h(\bar z)}{\partial z^2}
    \end{bmatrix}
                                                   \begin{bmatrix}
                                                     1 \\
                                                     \delta z
                                                   \end{bmatrix}
    \\
    \delta z =& z-\bar z.
  \end{align}
\end{subequations}
If $h : \R^n\to \R^m$ we form the quadratic approximation by stacking
all of the quadratic approximations of the entries:
\begin{equation}
  \label{eq:stackedQuad}
  \qApprox(h(z))_{\bar z} = [\qApprox(h_1(z))_{\bar
    z},\ldots,\qApprox(h_m(z))_{\bar z}]^\top
\end{equation}

Let
\begin{equation}
  \label{eq:zDef}
  z_k = \begin{bmatrix}
    x_k \\
    u_{:,k}
  \end{bmatrix}
\end{equation}
and let $\bar x_k$ and $\bar
u_{:,k}$ be a trajectory of states and actions satisfying the dynamic
equations from \eqref{eq:dynamics}.
The approximate Bellman recursion around this trajectory is given by:
\begin{subequations}
  \label{eq:bellmanApprox}
  \begin{align}
    \tilde V_{n,T+1}(x_{T+1}) &= 0 \\
    \label{eq:Qquad}
    \tilde Q_{n,k}(z_k) &= \qApprox(
                          c_{n,k}(z_k) + \tilde V_{n,k+1}(f_k(z_k))
                          )_{\bar z_k}\\
    \label{eq:quadQGame}
    \tilde V_{n,k}(x_k) &= \min_{u_{n,k}} \tilde Q_{n,k}(x_k,u_{:,k}).
  \end{align}
\end{subequations}

Note that \eqref{eq:quadQGame} is now a quadratic game in the $u_{:,k}$
variables which has unique and ready solution \cite{basar1999dynamic}. Recall the $\J(u)$ function defined in \eqref{eq:JFun}.
A sufficient condition for solvability of these games is
given in terms of $\J(u)$ is given in the following lemma.
Its proof is in Appendix~\ref{app:quadSol}.

\begin{lemma}
  \label{lem:quadSol}
  {\it
    If $\frac{\partial \J(\bar u)}{\partial u}$ is invertible, the game defined by
    \eqref{eq:quadQGame} has a unique solution of the form:
    \begin{equation}
      \label{eq:localSolution}
      u_{:,k} = \bar u_{:,k} +  \tilde K_k \delta x_k + \tilde s_k.
    \end{equation}
  }
\end{lemma}

In the notation defined above, we have
that $\delta x_k = x_k - \bar x_k$.
Note that if
$\frac{\partial \J(u^\star)}{\partial u}$ is invertible, then
$\frac{\partial \J(\bar
  u)}{\partial u}$ is invertible for all $\bar u$ in a neighborhood of $u^\star$.

Here we provide the DDP algorithm for applying DDP game solution in pseudo code.
\begin{algorithm}
  \caption{\label{alg:ddp} Differential Dynamic Programming for Nonlinear Dynamic Games}
  \begin{algorithmic}
    \State Generate an initial trajectory $\bar x, \bar u$
    \Loop
    \State{\textbf{Backward Pass:}}
    \State Perform the approximate Bellman recursion from \eqref{eq:bellmanApprox}
    \State Compute $\tilde K_k$ and $\tilde s_k$ from \eqref{eq:localSolution}.
    \State{\textbf{Forward Pass:}}
    \State{Generate a new trajectory using the affine policy defined
      by $\tilde K_k, \tilde s_k$}
    \EndLoop
  \end{algorithmic}
\end{algorithm}

\subsection{Implementation Details}
\label{sec:algMatrices}

All of the operations in the backwards pass of the DDP algorithm,
Algorithm~\ref{alg:ddp}, can be expressed more explicitly in terms of
matrices.

To construct the required matrices, we define the following
approximation terms:

\begin{subequations}
  \label{eq:approximations}
  \begin{align}
    & A_k = \frac{\partial f_k(x_k, u_{:,k})}{\partial x_k} \Big|_{\bar u},
      \quad B_{k} = \frac{\partial f_k(x_k, u_{:,k})}{\partial u_{:,k}} \Big|_{\bar u} \\
    & G_k^l =
      \begin{bmatrix}
        \frac{\partial^2 f^l_k}{\partial x_k^2} & \frac{\partial^2 f^l_k}{\partial x_k   \partial u_{:,k}} \\
        \frac{\partial^2 f^l_k}{\partial u_{:,k} \partial x_k} & \frac{\partial^2 f^l_k}{\partial u_{:,k}^2}
      \end{bmatrix} \Bigg|_{\bar{u}}, \  l = 1, 2, \ldots, n_x \\
    & R_k(\delta x_k, \delta u_{:,k}) =
      \begin{bmatrix}
        \begin{bmatrix}
          \delta x_k \\
          \delta u_{:,k}
        \end{bmatrix}^T G_k^0
        \begin{bmatrix}
          \delta x_k \\
          \delta u_{:,k}
        \end{bmatrix} \\
        \begin{bmatrix}
          \delta x_k \\
          \delta u_{:,k}
        \end{bmatrix}^T G_k^1
        \begin{bmatrix}
          \delta x_k \\
          \delta u_{:,k}
        \end{bmatrix} \\
        \vdots \\
        \begin{bmatrix}
          \delta x_k \\
          \delta u_{:,k}
        \end{bmatrix}^T G_k^{n_x}
        \begin{bmatrix}
          \delta x_k \\
          \delta u_{:,k}
        \end{bmatrix}
      \end{bmatrix}
  \end{align}
  \begin{align}
    M_{n, k} &= \left.
               \begin{bmatrix}
                 2 c_{n, k} & \frac{\partial c_{n, k}}{\partial x_k}
                 & \frac{\partial c_{n, k}}{\partial u_{:,k}} \\
                 \frac{\partial c_{n, k}}{\partial x_k}^\top
                 & \frac{\partial^2 c_{n, k}}{\partial x_k^2}
                 & \frac{\partial^2 c_{n, k}}{\partial x_k \partial u_{:,k}} \\
                 \frac{\partial c_{n, k}}{\partial u_{:,k}}^\top
                 & \frac{\partial^2 c_{n, k}}{\partial u_{:,k}\partial x_k}
                 & \frac{\partial^2 c_{n, k}}{\partial u_{:,k}^2}
               \end{bmatrix} \right\rvert_{\bar u}
                   \label{eq:cost_devs} \\
    \nonumber
             &=
               \begin{bmatrix}
                 M_{n, k}^{11} & M_{n, k}^{1x} & M_{n, k}^{1u} \\
                 M_{n, k}^{x1} & M_{n, k}^{xx} & M_{n, k}^{xu} \\
                 M_{n, k}^{u1} & M_{n, k}^{ux} & M_{n, k}^{uu}
               \end{bmatrix}.
  \end{align}
\end{subequations}

Using the notation from \eqref{eq:qApprox}, \eqref{eq:stackedQuad} and \eqref{eq:zDef},
the second-order approximations of the dynamics and
cost are given by:
\begin{subequations}
  \begin{align}
    \label{eq:dynQuad}
    & \qApprox(f_k(z_k))_{\bar z_k} = f_k(\bar z_k) +  A_k \delta x_k +
                                    B_k \delta u_{:,k} + R_k(\delta z_k) \\
    & \qApprox(c_{n,k}(z_k))_{\bar z_k} = \begin{bmatrix}
      1 \\
      \delta z_k
    \end{bmatrix}^\top M_{n,k}
    \begin{bmatrix}
      1 \\
      \delta z_k
    \end{bmatrix}.
  \end{align}
\end{subequations}

By construction $\tilde V_{n,k}(x_k)$ and $\tilde Q_{n,k}(x_k,u_{:,k})$ are
quadratic, and so there must be matrices $\tilde S_{n,k}$ and $\tilde
\Gamma_{n,k}$ such that
\begin{subequations}
  \label{eq:ddp_cost_matrices}
  \begin{align}
    \label{eq:ddpV}
    \tilde V_{n,k}(x_k) &= \frac{1}{2}
                          \begin{bmatrix}
                            1 \\
                            \delta x_k \\
                          \end{bmatrix}^\top \begin{bmatrix}
                            \tilde S_{n, k}^{11} & \tilde S_{n, k}^{1x} \\
                            \tilde S_{n, k}^{x1} & \tilde S_{n, k}^{xx}
                          \end{bmatrix}
                                                   \begin{bmatrix}
                                                     1 \\
                                                     \delta x_k \\
                                                   \end{bmatrix} \\
    \label{eq:ddpQ}
    \tilde Q_{n,k}(x_k, u_{:,k}) &= \frac{1}{2}
                                   \begin{bmatrix}
                                     1 \\
                                     \delta x_k \\
                                     \delta u_{:,k}
                                   \end{bmatrix}^\top \tilde\Gamma_{n, k}
    \begin{bmatrix}
      1 \\
      \delta x_k \\
      \delta u_{:,k}
    \end{bmatrix}.
  \end{align}
\end{subequations}

\begin{lemma}
  \label{lem:ddp_matrices}
  {\it
    The matrices in \eqref{eq:ddp_cost_matrices} are defined recursively
    by $\tilde S_{n,T+1} = 0$ and:
    \begin{subequations}
      \begin{align}
        \label{eq:TildeDDef}
        & \tilde D_{n,k} = \sum_{l=1}^{n_x} \tilde S_{n,k+1}^{1x^l}
          G_k^l
        \\
        & \tilde \Gamma_{n, k} = M_{n, k} \nonumber \\
        \label{eq:TildeGammaBackprop}
        &+
          \begin{bmatrix}
            \tilde S_{n,k+1}^{11} & \tilde S_{n,k+1}^{1x}A_k & \tilde S_{n,k+1}^{1x}B_k \\
            A_k^\top \tilde S_{n,k+1}^{x1} & A_k^\top \tilde S_{n,k+1}^{xx}A_k + \tilde D_{n, k}^{xx} & A_k^\top \tilde S_{n,k+1}^{xx} B_k + \tilde D_{n, k}^{xu} \\
            B_k^\top \tilde S_{n,k+1}^{x1}  & B_k^\top \tilde S_{n,k+1}^{xx}A_k + \tilde D_{n,k}^{ux} & B_k^\top \tilde S_{n,k+1}^{xx} B_k + \tilde D_{n, k}^{uu}
          \end{bmatrix} \\
        &=
          \begin{bmatrix}
            \tilde \Gamma_{n, k}^{11} & \tilde \Gamma_{n, k}^{1x} & \tilde \Gamma_{n, k}^{1u_1} & \tilde \Gamma_{n, k}^{1u_2} & \cdots & \tilde \Gamma_{n, k}^{1u_N} \\
            \tilde \Gamma_{n, k}^{x1} & \tilde \Gamma_{n, k}^{xx} & \tilde \Gamma_{n, k}^{xu_1} & \tilde \Gamma_{n, k}^{xu_2} & \cdots & \tilde \Gamma_{n, k}^{xu_N} \\
            \tilde \Gamma_{n, k}^{u_11} & \tilde \Gamma_{n, k}^{u_1x} & \tilde \Gamma_{n, k}^{u_1u_1} & \tilde \Gamma_{n, k}^{u_1u_2} & \cdots & \tilde \Gamma_{n, k}^{u_1u_N} \\
            \tilde \Gamma_{n, k}^{u_21} & \tilde \Gamma_{n, k}^{u_2x} & \tilde \Gamma_{n, k}^{u_2u_1} & \tilde \Gamma_{n, k}^{u_2u_2} & \cdots & \tilde \Gamma_{n, k}^{u_2u_N} \\
            \vdots & \vdots & \vdots & \vdots & \ddots & \vdots \\
            \tilde \Gamma_{n, k}^{u_N1} & \tilde \Gamma_{n, k}^{u_Nx} & \tilde \Gamma_{n, k}^{u_Nu_1} & \tilde \Gamma_{n, k}^{u_Nu_2} & \cdots & \tilde \Gamma_{n, k}^{u_Nu_N}
          \end{bmatrix} \\
        \label{eq:TildeFDef}
        & \tilde F_k =
          \begin{bmatrix}
            \tilde \Gamma_{1,k}^{u_1u} \\
            \tilde \Gamma_{2,k}^{u_2u} \\
            \vdots \\
            \tilde \Gamma_{N,k}^{u_Nu}
          \end{bmatrix} =
        \begin{bmatrix}
          \tilde \Gamma_{1,k}^{u_1u_1} & \tilde \Gamma_{1,k}^{u_1u_2} & \cdots & \tilde \Gamma_{1,k}^{u_1u_N} \\
          \tilde \Gamma_{2,k}^{u_2u_1} & \tilde \Gamma_{2,k}^{u_2u_2} & \cdots & \tilde \Gamma_{2,k}^{u_2u_N} \\
          \vdots & \vdots & \ddots & \vdots \\
          \tilde \Gamma_{N,k}^{u_Nu_1} & \tilde \Gamma_{N,k}^{u_Nu_2} & \cdots & \tilde \Gamma_{N,k}^{u_Nu_N}
        \end{bmatrix} \\
        \label{eq:TildePHsKDef}
        & \tilde P_k =
          \begin{bmatrix}
            \tilde \Gamma_{1,k}^{u_1x} \\
            \tilde \Gamma_{2,k}^{u_2x} \\
            \vdots \\
            \tilde \Gamma_{N,k}^{u_Nx}
          \end{bmatrix}, \quad
        \tilde H_k =
        \begin{bmatrix}
          \tilde \Gamma_{1,k}^{u_11} \\
          \tilde \Gamma_{2,k}^{u_21} \\
          \vdots \\
          \tilde \Gamma_{N,k}^{u_N1}
        \end{bmatrix} \\
        \label{eq:ddpStrategyMatrices}
        & \tilde s_k = - \tilde F_k^{-1} \tilde H_k, \quad \tilde K_k =
          - \tilde F_k^{-1} \tilde P_k \\
        \label{eq:TildeSDef}
        & \tilde S_{n, k} =
          \begin{bmatrix}
            1 & 0 & \tilde s_k^\top \\
            0 & I  & \tilde K_k^\top
          \end{bmatrix} \tilde \Gamma_{n, k}
                                   \begin{bmatrix}
                                     1 & 0 \\
                                     0 & I \\
                                     \tilde s_k & \tilde K_k
                                   \end{bmatrix} ,
      \end{align}
    \end{subequations}
    for $k=T,T-1,\ldots,0$.
  }
\end{lemma}

\begin{proof}
  By construction we must have $\tilde
  S_{n,T+1}=0$. Plugging \eqref{eq:dynQuad} into \eqref{eq:TildeGammaBackprop}
  and dropping all cubic and higher terms gives \eqref{eq:Qquad}. Since
  $u_{:,k} = \bar u_{:,k}+\delta u_{:,k}$ and $\bar u_{:,k}$ is
  constant, the static game defined in \eqref{eq:quadQGame} can be
  solved in the $\delta u_{:,k}$ variables. Differentiating
  \eqref{eq:ddpQ} by $\delta u_{n,k}$, collecting the derivatives for all players and setting them to zero leads to the
  necessary condition for an equilibrium:
  \begin{equation}
    \label{eq:ddp_necessary}
    \tilde F_k \delta u_{:,k} + \tilde P_k \delta x_k + \tilde H_k = 0.
  \end{equation}
  Thus, the matrices for the equilibrium strategy are given in
  \eqref{eq:ddpStrategyMatrices}. Plugging \eqref{eq:localSolution} into
  \eqref{eq:ddpQ} leads to \eqref{eq:TildeSDef}.
\end{proof}

\begin{remark}
  The next section will describe how the algorithm converges
  to strict Nash equilibria if it begins sufficiently close. To ensure
  that the algorithm converges
  regardless of initial condition, a Levenberg\--Marquardt style
  regularization can be employed. Such regularization has been used in
  centralized DDP algorithms, \cite{liao1991convergence,tassa2011theory}, to
  ensure that the required inverses exist and that the solution
  improves. In the current setting, such regularization would
  correspond to using a regularization of the form $\tilde F_k +
  \lambda I$ where $\lambda \ge 0$ is chosen sufficiently large to
  ensure that the matrix is positive definite. 
\end{remark}

\section{Convergence}
\label{sec:convergence}

This section outlines the convergence behavior of the DDP algorithm
for dynamic games. The main result is Theorem~\ref{thm:main} which
demonstrates quadratic convergence to local Nash equilibria: 
\begin{theorem}
  \label{thm:main}
  {\it
    If $u^\star$ is a strict local equilibrium such that
    $\frac{\partial \J(u^*)}{\partial u}$ is invertible, then
    the DDP algorithms converges locally to $u^{\star}$ at a quadratic
    rate.
  }
\end{theorem}

The proof
depends on several intermediate
results. Subsection~\ref{sec:newtonGame} reformulates Newton's method
for the necessary conditions, \eqref{problem:necessary}, as the
solution to a dynamic game. Subsection~\ref{sec:close} demonstrates
that the solutions of the dynamic games solved by Newton's method and
DDP close. Then, Subsection~\ref{sec:proof} finishes the convergence
proof by demonstrating that the DDP solution is sufficiently close to
the Newton solution to inherit its convergence property.

Throughout this section we will assume that both Newton's method and
DDP are starting from the same initial action trajectory, $\bar
u$. Let $u^N$ and $u^D$ be the updated action trajectories of Newton's
method and DDP, respectively. Define update steps, $\delta u^N$ and
$\delta u^D$, by:
\begin{equation}
  u^N = \bar u + \delta u^N \quad u^D = \bar u + \delta u^D.
\end{equation}
Additionally, we will assume that $u^\star$ is a strict local
equilibrium with $\frac{\partial \J(u^*)}{\partial u}$ invertible.

\subsection{Dynamic Programming Solution for the Newton Step}
\label{sec:newtonGame}
The proof of Theorem~\ref{thm:main} proceeds by demonstrating that the solutions from DDP and
Newton's method are sufficiently close that DDP inherits the quadratic
convergence of Newton's method.
To show closeness, we demonstrate that the Newton step can be
interpreted as the solution to a dynamic game. This dynamic game has a
recursive solution that is structurally similar to the recursions from
DDP. This subsection derives the corresponding game and solution.

The Newton step for solving \eqref{problem:necessary} is given by:
\begin{align}
  \label{eq:direct_newton}
  \frac{\partial \mathcal{J}(\bar u)}{\partial u} \delta u^N = - \mathcal{J}(\bar u).
\end{align}

This rule leads to a quadratic convergence to a root in
\eqref{problem:necessary} whenever $\nabla_u
\mathcal{J}(u)$ is locally Lipschitz and invertible
\cite{nocedal2006numerical}. The next two lemmas give game-theoretic
interpretations of the Newton step.

\begin{lemma}
  {\it
    Solving \eqref{eq:direct_newton} is equivalent to solving the
    quadratic game defined by:
    \begin{align}
      \label{eq:direct_game}
      \min_{\delta u_{n,:}} J_n(\bar u) +
      \frac{\partial J_n(\bar u)}{\partial u} \delta u
      + \frac{1}{2} \delta u^T \frac{\partial^2 J_n(\bar u)}{\partial u^2} \delta u
    \end{align}
  }
\end{lemma}
\begin{proof}
  Under the strict local equilibrium assumptions,
  \eqref{eq:direct_game} has a unique solution which is found by
  differentiating with respect to $\delta u_{n,:}$ and setting the
  result to $0$. Stacking these equations leads precisely to
  \eqref{eq:direct_newton}.
\end{proof}

The next lemma shows that \eqref{eq:direct_game} can be expressed as a
quadratic dynamic game.
\if\MODE\ARXIV{
  It is proved in Appendix~\ref{app:newton_game}.
}
\fi

\begin{lemma}
  \label{lem:newton_game}
  {\it
    The quadratic game defined in \eqref{eq:direct_game} is equivalent
    to the dynamic game defined by:
    \begin{subequations}
      \label{eq:newton_dp}
      \begin{align}
        \label{eq:newton_dp_objective}
        & \min_{u_{n,:}}\frac{1}{2} \sum_{k=0}^T \left(
          \begin{bmatrix}
            1 \\
            \delta x_k \\
            \delta u_{:,k}
          \end{bmatrix}^T
        M_{n, k}
        \begin{bmatrix}
          1 \\
          \delta x_k \\
          \delta u_{:,k}
        \end{bmatrix}
        + M^{1x}_{n,k} \Delta x_k \right)
      \end{align}
      subject to
      \begin{align}
        & \quad \quad \delta x_0 = 0 \\
        \label{eq:newton_dp_init1}
        & \quad \quad \Delta x_0 = 0 \\
        \label{eq:newton_dp_dynamics0}
        & \quad \quad \delta x_{k+1} = A_k \delta{x}_k + B_k \delta u_{:,k} \\
        \label{eq:newton_dp_dynamics1}
        & \quad \quad \Delta x_{k+1} = A_k \Delta x_k + R_k(\delta x_k, \delta u_{:,k})\\
        & \quad \quad k = 0, 1, \ldots, T
      \end{align}
    \end{subequations}
  }
\end{lemma}

Note that the states of the dynamic game are given by $\delta x_k$ and
$\Delta x_k$ as
\begin{subequations}
  \label{eq:newton_dp_states}
  \begin{align}
      \label{eq:newton_dp_states0}
      & \delta x_k = \sum_{i = 0}^{T} \frac{\partial x_k}{\partial u_{:,i}} \Big|_{\bar x, \bar{u}} \delta u_{:,i} \\
      \label{eq:newton_dp_states1}
      & \Delta x_k^l = \sum_{i = 0}^{T} \sum_{j = 0}^{T} \delta u_{:,i}^\top {\frac{\partial^2 x_k^l}{\partial u_{:,i} \partial u_{:,j}}} \Big|_{\bar x,\bar{u}} \delta u_{:,j}, \  l = 1, 2, \ldots, n_x
  \end{align}
\end{subequations}

\if\MODE\ACC{
\paragraph*{Proof Sketch}
We only give an outline for the proof here due to space constraints. The states dynamics \eqref{eq:newton_dp_dynamics0}, \eqref{eq:newton_dp_dynamics1} can be proven inductively from the definitions \eqref{eq:newton_dp_states} and original system dynamics quadratic approximation \label{eq:dynQuad}.

  To prove that \eqref{eq:newton_dp_objective} is equivalent to \eqref{eq:direct_game}, we'll need explicit expressions for each of the associated derivative $\frac{\partial J_n(u)}{\partial u_{:,i}}$ and $\frac{\partial^2 J_n(u)}{\partial u_{:,i} \partial u_{:,j}}$ computed from \eqref{eq:cost} and \eqref{eq:dynamics} in terms of derivatives of $c_{n,k}(x_k, u_k)$. The fact that current cost or dynamics does not depend on future inputs is utilized in the derivation. Expression in the form of summation over time $k$ for each additional term in \eqref{eq:direct_game} is then computed via summing up the corresponding first or second order derivatives of $J_n(u)$ w.r.t. $u_i$ and/or $u_j$. It can be easily verified that \eqref{eq:newton_dp_objective} is the same as \eqref{eq:direct_game} in its summation over time form.
\hfill\QED
}
\fi

It follows that the equilibrium solution of this dynamic
game is characterize by the following Bellman recursion:
\begin{subequations}
  \label{eq:newtonRecursion}
  \begin{align}
    V_{n,T+1}(\delta x_{T+1},\Delta x_{T+1}) &= 0 \\
    \nonumber
    Q_{n,k}(\delta x_k,\Delta x_k,\delta u_{:,k})
                                             &=
    \\
                                             &
                                               \hspace{-8em}
                                               \frac{1}{2}
                                               \left(
                                               \begin{bmatrix}
                                                 1 \\
                                                 \delta x_k \\
                                                 \delta u_{:,k}
                                               \end{bmatrix}^\top
    M_{n,k}
    \begin{bmatrix}
      1 \\
      \delta x_k \\
      \delta u_{:,k}
    \end{bmatrix}
    +M^{1x}_{n,k} \Delta x_k \right) \\
                                             &
                                               \hspace{-8em}
                                               + V_{n,k+1}(A_k\delta x_k +B_k
                                               \delta u_{:,k},
                                               A_k \Delta x_k + R_k(\delta x_k,\delta u_{:,k})) \\
    \label{eq:newtonQgame}
    V_{n,k}(\delta x_k,\Delta x_k) &= \min_{
                                     \delta u_{n,k}} Q_{n,k}(\delta
                                     x_k,\Delta x_k,\delta u_{:,k}).
  \end{align}
\end{subequations}
Note that \eqref{eq:newtonQgame} defines a static quadratic game and $V_{n,k}(\delta x_k, \Delta x_k)$ is found by solving the game and substituting the solution back to $Q_{n,k}(\delta x_k, \Delta x_k, \delta u_{:,k})$.

The next lemma describes an explicit solution to the backward
recursion \eqref{eq:newtonRecursion}. The key step in the convergence
proof is showing that the matrices used in this recursion are
appropriately close to the matrices used in DDP.

\begin{lemma}
  \label{lem:newtonMatrices}
  {\it
    The functions $V_{n,k}$ and $Q_{n,k}$ can be expressed as
    \begin{subequations}
    \label{eq:newton_dp_solution_state_action_val_fun}
      \begin{align}
      \label{eq:newton_dp_solution_state_val_fun}
        & V_{n,k}(\delta x_k, \Delta x_k) = \frac{1}{2} \left(
          \begin{bmatrix}
            1 \\
            \delta x_k
          \end{bmatrix}^T S_{n,k}
        \begin{bmatrix}
          1 \\
          \delta x_k
        \end{bmatrix} + \Omega_{n,k} \Delta x_k \right)
        \\
        & \hspace{-.8em} Q_{n,k}(\delta x_k, \Delta x_k, \delta u_{:,k}) = \frac{1}{2} \left(
          \begin{bmatrix}
            1 \\
            \delta x_k \\
            \delta u_{:,k}
          \end{bmatrix}^T \Gamma_{n,k}
        \begin{bmatrix}
          1 \\
          \delta x_k \\
          \delta u_{:,k}
        \end{bmatrix} + \Omega_{n,k} \Delta x_k \right)
      \end{align}
    \end{subequations}
    where the matrices $S_{n,k}$, $\Gamma_{n,k}$, and $\Omega_{n,k}$
    are defined recusrively by $S_{n,T+1}=0$, $\Omega_{n,T+1} = 0$,
    and
    \begin{subequations}
      \label{eq:newton_dp_matrices}
      \begin{align}
        \label{eq:newton_dp_solution_omega1}
        & \Omega_{n,k} = M_{n,k}^{1x} + \Omega_{n,k+1} A_k \\
        \label{eq:newton_dp_solution_D}
        & D_{n,k} = \sum_{l=1}^n \Omega_{n,k+1}^l G_k^l  \\
        \label{eq:GammaDef}
        & \Gamma_{n,k} = M_{n,k} \nonumber \\
        & +
          \begin{bmatrix}
            S_{n, k+1}^{11} & S_{n,k+1}^{1x}A_k & S_{n,k+1}^{1x}B_k \\
            A_k^\top S_{n,k+1}^{x1} & A_k^\top S_{n,k+1}^{xx}A_k + D_k^{xx} & A_k^\top S_{n,k+1}^{xx} B_k + D_k^{xu} \\
            B_k^\top  S_{n,k+1}^{x1}  & B_k^\top  S_{n,k+1}^{xx}A_k + D_k^{ux} & B_k^\top S_{n,k+1}^{xx} B_k + D_k^{uu}
          \end{bmatrix} \\
        & \quad \ =
          \begin{bmatrix}
            \Gamma_{n, k}^{11} & \Gamma_{n, k}^{1x} & \Gamma_{n, k}^{1u_1} & \Gamma_{n, k}^{1u_2} & \cdots & \Gamma_{n, k}^{1u_N} \\
            \Gamma_{n, k}^{x1} & \Gamma_{n, k}^{xx} & \Gamma_{n, k}^{xu_1} & \Gamma_{n, k}^{xu_2} & \cdots & \Gamma_{n, k}^{xu_N} \\
            \Gamma_{n, k}^{u_11} & \Gamma_{n, k}^{u_1x} & \Gamma_{n, k}^{u_1u_1} & \Gamma_{n, k}^{u_1u_2} & \cdots & \Gamma_{n, k}^{u_1u_N} \\
            \Gamma_{n, k}^{u_21} & \Gamma_{n, k}^{u_2x} & \Gamma_{n, k}^{u_2u_1} & \Gamma_{n, k}^{u_2u_2} & \cdots & \Gamma_{n, k}^{u_2u_N} \\
            \vdots & \vdots & \vdots & \vdots & \ddots & \vdots \\
            \Gamma_{n, k}^{u_N1} & \Gamma_{n, k}^{u_Nx} & \Gamma_{n, k}^{u_Nu_1} & \Gamma_{n, k}^{u_Nu_2} & \cdots & \Gamma_{n, k}^{u_Nu_N}
          \end{bmatrix}
      \end{align}
      \begin{align}
        \label{eq:newton_dp_invert}
        & F_k =
          \begin{bmatrix}
            \Gamma_{1k}^{u_1u} \\
            \Gamma_{2k}^{u_2u} \\
            \vdots \\
            \Gamma_{Nk}^{u_Nu}
          \end{bmatrix} =
        \begin{bmatrix}
          \Gamma_{1k}^{u_1u_1} & \Gamma_{1k}^{u_1u_2} & \cdots & \Gamma_{1k}^{u_1u_N} \\
          \Gamma_{2k}^{u_2u_1} & \Gamma_{2k}^{u_2u_2} & \cdots & \Gamma_{2k}^{u_2u_N} \\
          \vdots & \vdots & \ddots & \vdots \\
          \Gamma_{Nk}^{u_Nu_1} & \Gamma_{Nk}^{u_Nu_2} & \cdots & \Gamma_{Nk}^{u_Nu_N}
        \end{bmatrix} \\
        & P_k =
          \begin{bmatrix}
            \Gamma_{1k}^{u_1x} \\
            \Gamma_{2k}^{u_2x} \\
            \vdots \\
            \Gamma_{Nk}^{u_Nx}
          \end{bmatrix}, \quad
        H_k =
        \begin{bmatrix}
          \Gamma_{1k}^{u_11} \\
          \Gamma_{2k}^{u_21} \\
          \vdots \\
          \Gamma_{Nk}^{u_N1}
        \end{bmatrix} \\
        \label{eq:newton_strategy}
        & s_k = - F_k^{-1} H_k, \quad K_k = - F_k^{-1} P_k \\
        \label{eq:newton_dp_solution_S}
        & S_{n, k} =
          \begin{bmatrix}
            1 & 0 & s_k^\top \\
            0 & I & K_k^\top
          \end{bmatrix} \Gamma_{n, k}
                           \begin{bmatrix}
                             1 & 0 \\
                             0 & I \\
                             s_k & K_k
                           \end{bmatrix}
      \end{align}
    \end{subequations}
    for $k = T,T-1,\ldots,0$.
  }
\end{lemma}

From this lemma, we can see that the matrices used in the recursions
for both DDP and Newton's method are very similar in
structure. Indeed, the iterations are identical aside from the
definitions of the $D_{n,k}$ and $\tilde D_{n,k}$ matrices.

\begin{proof}
  The proof is very similar to the proof for Lemma \ref{lem:ddp_matrices}. Solving the equilibrium strategy and $V_{n,k}(\delta x_k, \Delta x_k)$ based on $Q_{n,k}(\delta x_k, \Delta x_k, \delta u_{:,k})$ is the same as how we arrived at \eqref{eq:ddp_necessary} and \eqref{eq:TildeSDef}, since the extra terms of $\Delta x_k$ are not coupled with $\delta u_{:,k}$ and other terms are of the exact same form. 
  The stepping back in time of $Q_k(\delta x_k, \Delta x_k, \delta u_{:,k})$ is acheived by substituting \eqref{eq:newton_dp_dynamics0} and \eqref{eq:newton_dp_dynamics1} into \eqref{eq:newton_dp_solution_state_val_fun}, which is slightly different because of the extra terms related to $\Delta x_k$.
  \begin{subequations}
    \label{eq:newton_dp_backprop}
    \begin{align}
      & Q_{n,k}(\delta x_k, \Delta x_k, \delta u_{:,k}) \\
      &= \frac{1}{2} \left(
        \begin{bmatrix}
          1 \\
          \delta x_k \\
          \delta u_{:,k}
        \end{bmatrix}^T
      M_{n,k}
      \begin{bmatrix}
        1 \\
        \delta x_k \\
        \delta u_{:,k}
      \end{bmatrix} + M^{1k}_{n,k} \Delta x_k \right) \nonumber \\
      & \quad + V_{n,k+1}(A_k \delta{x}_k + B_k  \delta{u}_k, A_k \Delta x_{k}  + R_k(\delta x_k, \delta u_{:,k})) \\
      &= \frac{1}{2} \left(
        \begin{bmatrix}
          1 \\
          \delta x_k \\
          \delta u_{:,k}
        \end{bmatrix}^T
      M_{n,k}
      \begin{bmatrix}
        1 \\
        \delta x_k \\
        \delta u_{:,k}
      \end{bmatrix} \right) + \nonumber \\
      & \frac{1}{2}
        \begin{bmatrix}
          1 \\
          \delta{x}_k \\
          \delta u_{:,k}
        \end{bmatrix}^T
      \begin{bmatrix}
        S^{11}_{n,k+1} & S^{1x}_{n,k+1} A_k & S^{1x}_{n,k+1} B_k \\
        A^\top_k S^{x1}_{n,k+1} & A^\top_k S^{xx}_{n,k+1} A_k & A^\top_k S^{xx}_{n,k+1} B_k \\
        B^\top_k S^{x1}_{n,k+1} & B^\top_k S^{xx}_{n,k+1} A_k & B^\top_k S^{xx}_{n,k+1} B_k
      \end{bmatrix}
                                                                \begin{bmatrix}
                                                                  1 \\
                                                                  \delta{x}_k \\
                                                                  \delta u_{:,k}
                                                                \end{bmatrix} \nonumber \\
      & \quad + \frac{1}{2} \left( (M^{1k}_{n,k} + \Omega_{n,k+1} A_k) \Delta x_{k} \right) \nonumber \\
      & \quad + \frac{1}{2} \left(
        \begin{bmatrix}
          \delta x_k \\
          \delta u_{:,k}
        \end{bmatrix}^T
      D_{n,k}
      \begin{bmatrix}
        \delta x_k \\
        \delta u_{:,k}
      \end{bmatrix} \right)
    \end{align}
  \end{subequations}
  So (\ref{eq:newton_dp_solution_omega1}), (\ref{eq:newton_dp_solution_D}) and (\ref{eq:GammaDef}) are true.
\end{proof}

\subsection{Closeness Lemmas}
\label{sec:close}
This subsection gives a few lemmas which imply that the Newton step,
$\delta u^N$, and the DDP step, $\delta u^D$, are close. For the rest
of the section, we set $\norm{\bar u - u^{\star}} = \epsilon$.

The following lemma shows that the matrices used in the backwards
recursion are close. It is proved in Appendix~\ref{app:matrixClose}

\begin{lemma}
  \label{lem:matrixClose}
  {\it
    The matrices from the backwards recursions of DDP and Newton's method are close in the following sense:
    \begin{subequations}
      \begin{align}
        \nonumber
        &\tilde \Gamma_{n,k} - \Gamma_{n,k}  \\
        &=
          \begin{bmatrix}
            \tilde \Gamma_{n,k}^{11} & \tilde \Gamma_{n,k}^{1x} & \tilde \Gamma_{n,k}^{1u} \\
            \tilde \Gamma_{n,k}^{x1} & \tilde \Gamma_{n,k}^{xx} & \tilde \Gamma_{n,k}^{xu} \\
            \tilde \Gamma_{n,k}^{u1} & \tilde \Gamma_{n,k}^{ux} & \tilde \Gamma_{n,k}^{uu}
          \end{bmatrix}
                                                                  -
                                                                  \begin{bmatrix}
                                                                    \Gamma_{n,k}^{11} & \Gamma_{n,k}^{1x} & \Gamma_{n,k}^{1u} \\
                                                                    \Gamma_{n,k}^{x1} & \Gamma_{n,k}^{xx} & \Gamma_{n,k}^{xu} \\
                                                                    \Gamma_{n,k}^{u1} & \Gamma_{n,k}^{ux} & \Gamma_{n,k}^{uu}
                                                                  \end{bmatrix}
        \\
        &=
          \begin{bmatrix}
            O(\epsilon) & O(\epsilon^2) & O(\epsilon^2) \\
            O(\epsilon^2) & O(\epsilon) & O(\epsilon) \\
            O(\epsilon^2) & O(\epsilon) & O(\epsilon)
          \end{bmatrix},
      \end{align}
      \begin{align}
        \tilde S_{n,k} - S_{n,k} = &
                                     \begin{bmatrix}
                                       \tilde S_{n,k}^{11} & \tilde S_{n,k}^{1x} \\
                                       \tilde S_{n,k}^{x1} & \tilde S_{n,k}^{xx}
                                     \end{bmatrix}
                                                             -
                                                             \begin{bmatrix}
                                                               S_{n,k}^{11} & S_{n,k}^{1x} \\
                                                               S_{n,k}^{x1} & S_{n,k}^{xx}
                                                             \end{bmatrix} \\
        =&
           \begin{bmatrix}
             O(\epsilon) & O(\epsilon^2) \\
             O(\epsilon^2) & O(\epsilon)
           \end{bmatrix} ,
      \end{align}
      \begin{equation}
        \begin{bmatrix}
          \tilde s_{n,l} & \tilde K_{n,k}
        \end{bmatrix}
        -
        \begin{bmatrix}
          s_{n,l} & K_{n,k}
        \end{bmatrix} =
        \begin{bmatrix}
          O(\epsilon^2) & O(\epsilon)
        \end{bmatrix}.
      \end{equation}
      Furthermore, the following matrices are small:
      \begin{align}
        \begin{bmatrix}
          \Gamma_{n,k}^{1x} & \Gamma_{n,k}^{1,u}
        \end{bmatrix} & = O(\epsilon),
        & \begin{bmatrix}
          \tilde \Gamma_{n,k}^{1x} & \tilde \Gamma_{n,k}^{1,u}
        \end{bmatrix} = O(\epsilon)
        \\
        s_{n,k} &= O(\epsilon),
                      & \tilde s_{n,k} = O(\epsilon).
      \end{align}
    \end{subequations}
  }
\end{lemma}

After showing that the matrices are close, it can be shown that the
states and actions computed in the update steps are close. It is
proved in Appendix~\ref{app:solClose}.

\begin{lemma}
  \label{lem:solClose}
  {\it
    The states and actions computed by DDP and Newton's method
    are close:
    \begin{subequations}
      \begin{align}
        \delta u_{:,k}^D - \delta u_{:,k}^N &= O(\epsilon^2) \\
        \delta x_k^D - \delta x_k^N &= O(\epsilon^2)
      \end{align}
      Furthermore, the updates are small:
      \begin{align}
        \delta u_{:,k}^D & =O(\epsilon ), & \delta u_{:,k}^N &= O(\epsilon) \\
        \delta x_k^D &= O(\epsilon), & \delta x_k^N &= O(\epsilon).
      \end{align}
    \end{subequations}
  }
\end{lemma}

\subsection{Proof of Theorem~\ref{thm:main}}
\label{sec:proof}
Lemma~\ref{lem:solClose} implies that $\|\delta u^N - \delta u^D\| =
O(\epsilon^2)$. Furthermore, the Newton step satisfies:
\begin{equation}
  \|\bar u + \delta u^N - u^\star\| = O(\epsilon^2).
\end{equation}
See~\cite{nocedal2006numerical}.
The proof of quadratic convergence is completed by the following steps:
\begin{subequations}
  \begin{align}
    \|\bar u + \delta u^D - u^{\star}\| &= \| \bar u + \delta u^N - u^{\star} + \delta
                                          u^D - \delta u^N \| \\
                                        &\le \|\bar u + \delta u^N- u^{\star}\| +
                                          \|\delta u^D - \delta u^N\| \\
                                        &= O(\epsilon^2).
  \end{align}
\end{subequations}
\hfill \QED

\section{Numerical Example}

\label{sec:example}

We apply the proposed DDP algorithm for deterministic nonlinear dynamic games to a toy examples in this section. The example is impletmented in Python and all derivatives of nonlinear functions are computed via Tensorflow \cite{tensorflow2015-whitepaper}.

\begin{figure}[!t]
  \centering
  \subfigure[state]
    {
      \includegraphics[width=1.5in]{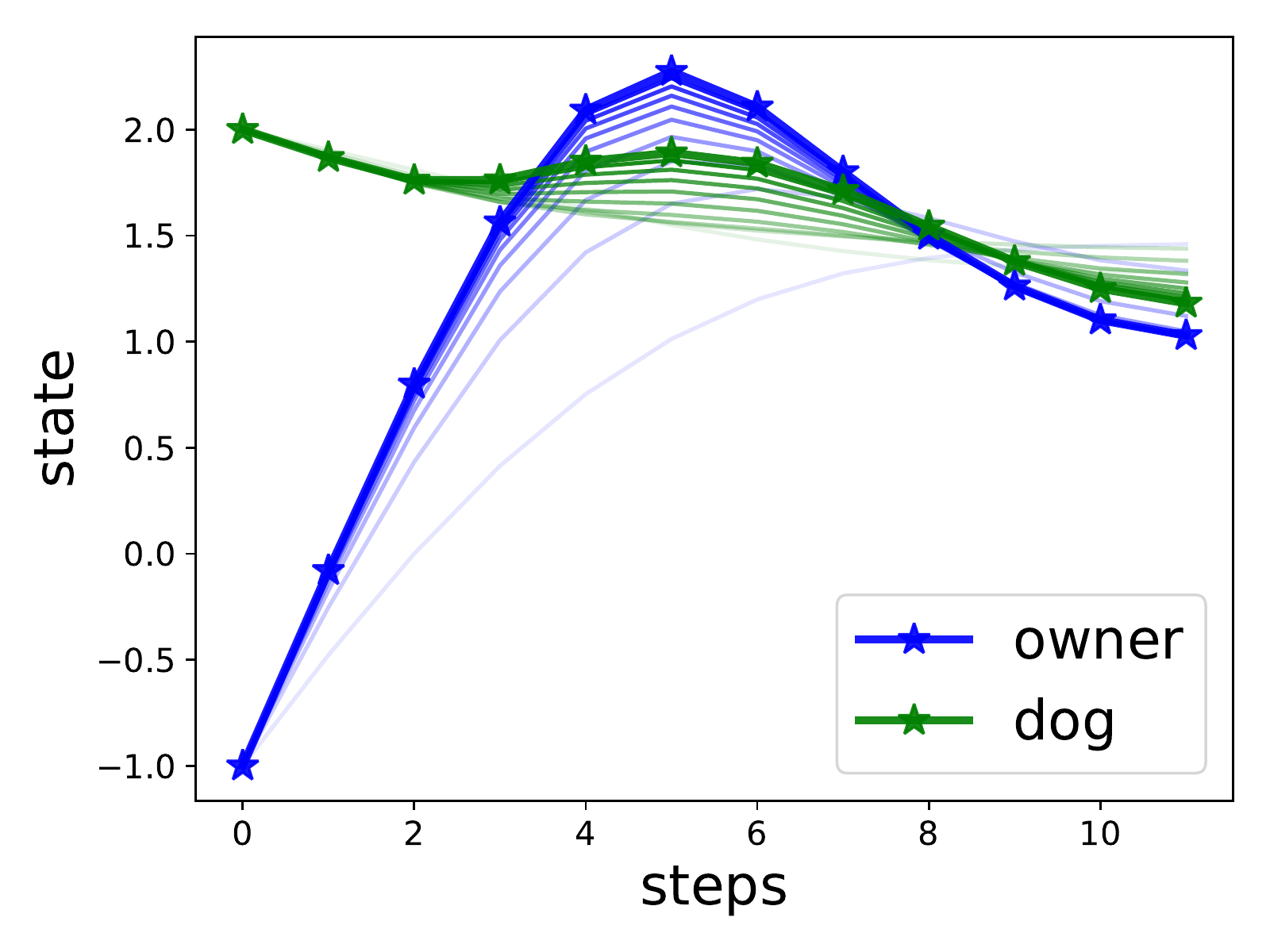}
      \label{fig:od_state}
    }
    \subfigure[state]
    {
      \includegraphics[width=1.5in]{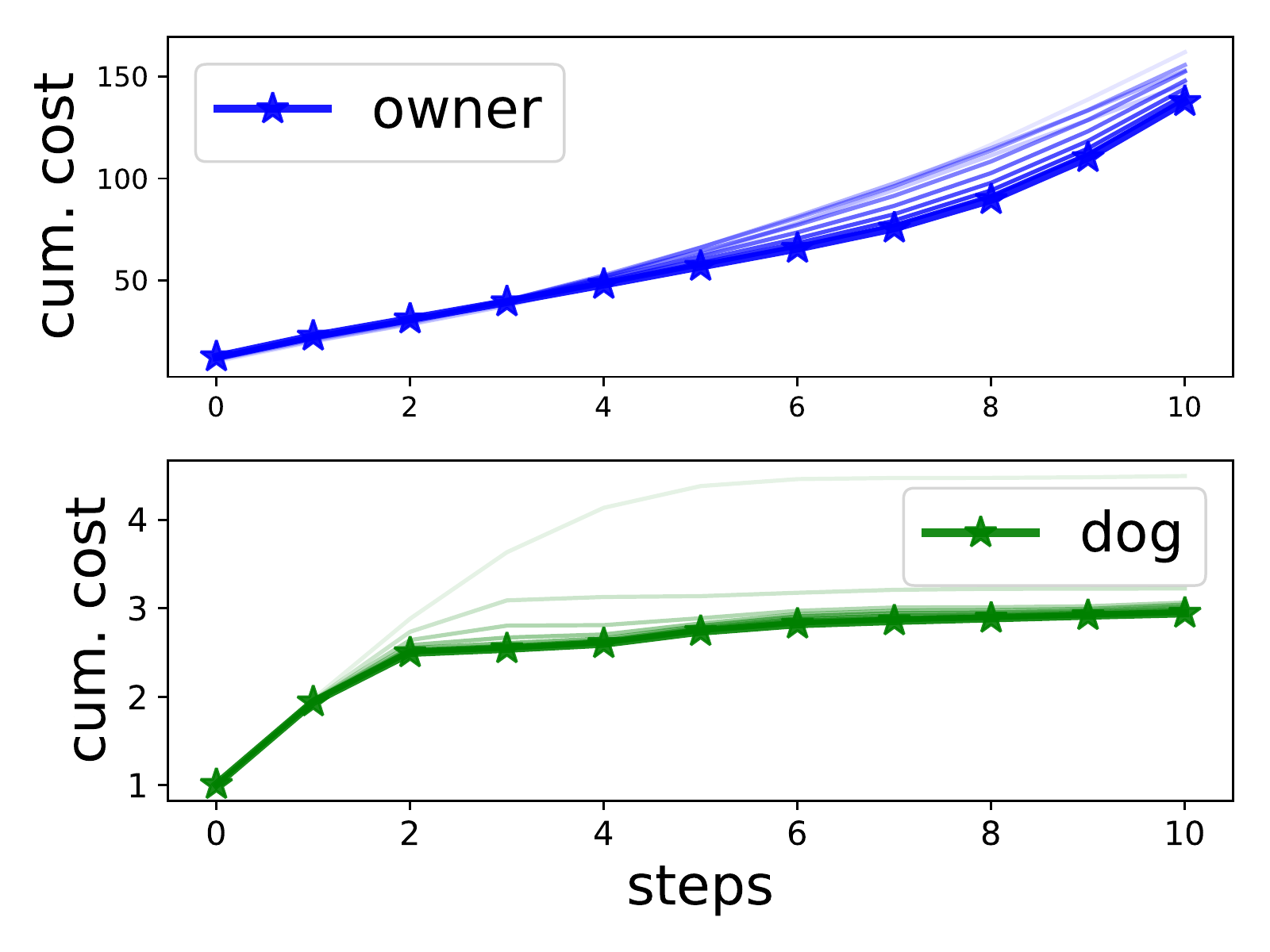}
      \label{fig:od_ccost}
    }
  \caption{Owner-dog dynamic game. In order to keep the dog around $x^1 = 2$, the owner has to overshoot and then come back to $x^0=1$. The dog learns to get closer to the owner over iterations, which is what we would expect given how the problem is formulated. As can be seen in Fig. \ref{fig:od_ccost}, the cumulative costs for both players reduces over iterations.}
  \label{fig:od}
\end{figure}

\begin{figure}[!t]
  \centering
  \includegraphics[width=6cm]{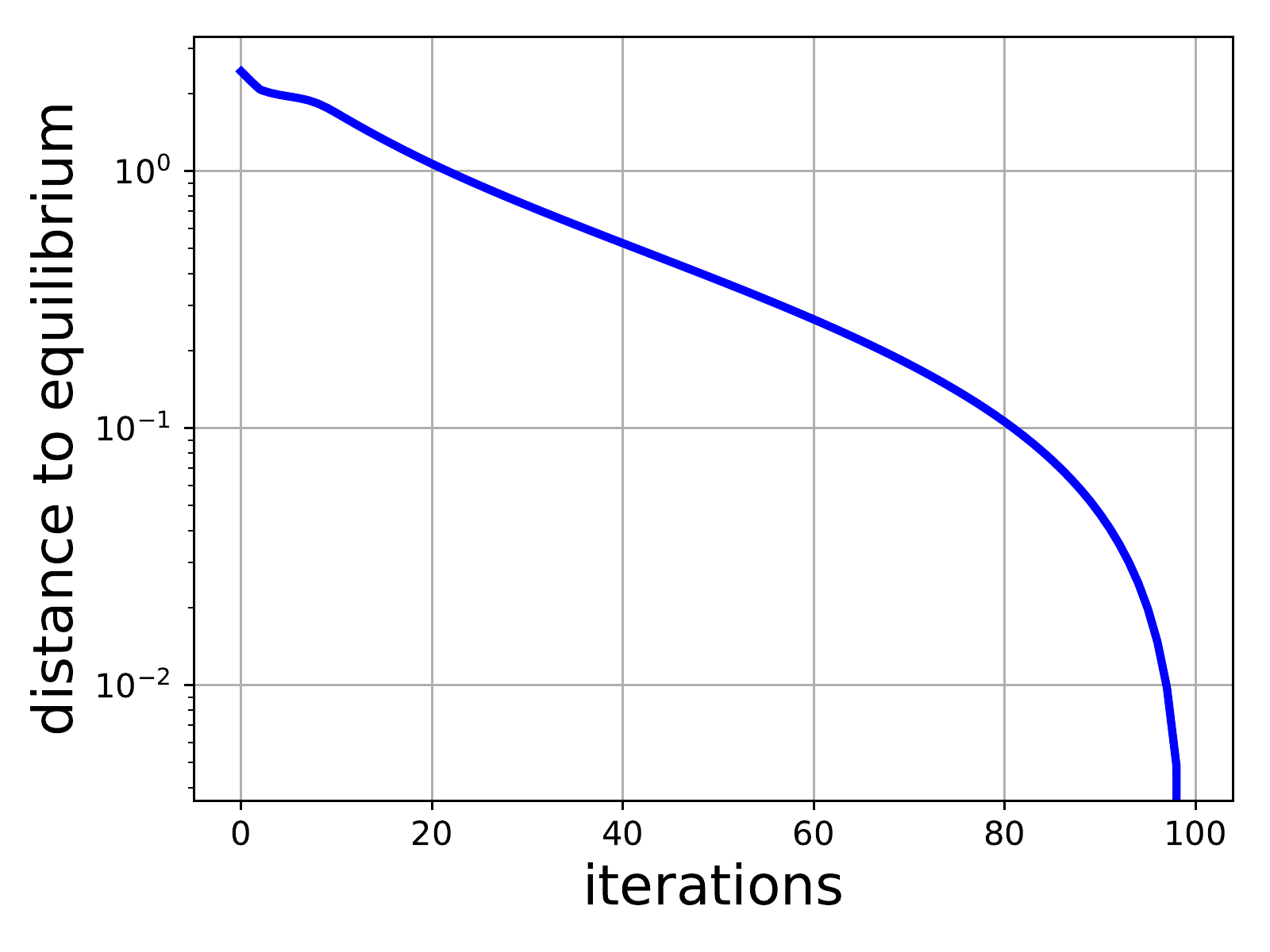}
  \caption{This shows the 2-norm distance between inputs $\bar u$ and the final equilibrium $u^{\star}$ over iterations. As can be seen that the error reduces sub-linearly on a log scaled plot, which is evidence that the algorithm converges quadratically.}
  \label{fig:od_convergence}
\end{figure}

We consider a simple 1-D owner-dog problem, with horizon $T=11$ and initial state $x_{:,0} = [-1, 2]$ where the dynamics of both the owner and the dog are given respectively by
\begin{subequations}
  \begin{align}
    x^0_{k+1} = x^0_{k} + \tanh u^0_k \\
    x^1_{k+1} = x^1_{k} + \tanh u^1_k
  \end{align}
\end{subequations}

The owner cares about going to $x^0 = 1$ and that the dog can stay at $x^1 = 2$. The dog, however, only tries to catch up with the owner. Each player also concerns itself with the energy consumption, therefore has a cost term related to the magnitude of its input. Their cost functions are formulated as
\begin{subequations}
  \begin{align}
    & c_{0,k}(x, u) = 10 \ \sigmoid ((x^0_k - 1)^2) + 40 (x^1_k - 2)^2 + (u^0_k)^2 \\
    & c_{1,k}(x, u) = \tanh^2(x^0_k - x^1_k) + (u^1_k)^2
  \end{align}
\end{subequations}

Nonlinear functions are added to the dynamics and costs to create a nonlinear game rather than for explicit physical meaning. We initialize a trajectory with zero input and initial state, i.e. $\bar u = [0., 0., \ldots, 0.]$ and $\bar x = [-1, 2, -1, 2, \ldots, -1, 2]$. We used an identity regularization matrix with a magnitude of 400.

Fig. \ref{fig:od} shows the solution via DDP to this problem over iterations, where the more transparent the trajectories, the earlier in the iterations they are. The starred trajectory is the final equilibrium solution. We simulated 100 iterations after the initial trajectory and picked 10 uniformly spaced ones to show in the figures.

\section{Conclusion}
\label{sec:conclusion}

In this paper we have shown how differential dynamic programming
extends to dynamic games. The key steps were involved finding explicit
forms for both DDP and Newton iterations that enable clean comparison
of their solutions. We demonstrated the performance of the algorithm
on a simple nonlinear dynamic game.

Many extensions are possible. We will examine larger examples and work
on numerical scaling. Also of interest are stochastic dynamic games
and problems in which agents have differing, imperfect information
sets. Additionally, handling scenarios in which agents have imperfect
model information will of great practical importance. 

\bibliographystyle{IEEEtran}
\bibliography{ref}

\appendices
\section{Proofs of Lemmas}
We first derive a few useful results which we'll use later in this chapter.

First, we bound the spectral radius of $F_k^{-1}$ from above
\begin{align}
  \rho(F_k^{-1}) \leq \hat F
\end{align}
where $\hat F$ is a constant. Consider inverting $\nabla_{u} \mathcal{J}(u)$ in Newton's method by successively eliminating $\delta u_{:,k}$ for $k=T,T-1,\ldots,0$. The
$F_k$ matrices are exactly the matrices which would be inverted when eliminating $\delta u_{:,k}$. Since $\nabla_u \mathcal{J}(u)$ is Lipschitz continuous, its eigenvalues are bounded away from zero in a neighborhood of $u^{\star}$. It follows that the eigenvalues of $F_K$ must also be bounded away from zero and $F_k^{-1}$ is bounded above.

Secondly, we assert that
\begin{align}
  \label{eq:omega_equiv}
  \Omega_{n,k} = \frac{\partial}{\partial x_k}\Big|_{\bar x, \bar u} \sum_{i=k}^{T} c_{n,i}(x_i, u_{:,i})
\end{align}
which means $\Omega_{n,k}$ captures the first-order effect of $x_k$ on the sum of all later costs for each player. $\Omega_{n,k}$ is constructed according to (\ref{eq:newton_dp_solution_omega1}). Equation (\ref{eq:omega_equiv}) is true for $k = T$ by construction. We proof by induction and assume that (\ref{eq:omega_equiv}) holds for $k+1$, i.e. $\Omega_{n,k+1} = \frac{\partial}{\partial x_{k+1}}\Big|_{\bar x, \bar u} \sum_{i={k+1}}^{T} c_{n,i}(x_i, u_{:,i})$, then
\begin{subequations}
  \begin{align}
    & \Omega_{n,k} = M_{n,k} + \Omega_{n, k+1} A_k \\
    & = \frac{\partial c_{n,k}(x_k, u_{:,k})}{x_k}\Big|_{\bar x, \bar u} + \left( \frac{\partial}{\partial x_{k+1}} \sum_{i=k+1}^{T} c_{n,i}(x_i, u_{:,i}) \right) \frac{\partial x_{k+1}}{\partial x_{k}} \Big|_{\bar x, \bar u} \\
    & = \frac{\partial c_{n,k}(x_k, u_{:,k})}{x_k}\Big|_{\bar x, \bar u} + \left( \frac{\partial}{\partial x_{k}}\Big|_{\bar x, \bar u} \sum_{i=k+1}^{T} c_{n,i}(x_i, u_{:,i}) \right) \\
    & = \frac{\partial}{\partial x_k}\Big|_{\bar x, \bar u} \sum_{i=k}^{T} c_{n,i}(x_i, u_{:,i})
  \end{align}
\end{subequations}
Therefore, (\ref{eq:omega_equiv}) holds for $k$. And by induction, all $k = 0,1, \ldots, T$.
\hfill \QED

The third useful result is that
\begin{align}
  \label{eq:J_grad_small}
  M^{1u}_{n,k} + \Omega_{n,k+1} B_k = \frac{\partial J_n(u)}{\partial u_{:,k}}\Big|_{\bar u} = O(\epsilon)
\end{align}
$\frac{\partial J_n(u)}{\partial u_{:,k}}\Big|_{\bar u} = O(\epsilon)$ is true because $J_n(u)$ is twice differentiable hence Lipschitz, i.e.
\begin{align}
  \Big|\Big| \frac{J_n(u)}{\partial u_{:,k}}\Big|_{\bar u} - \frac{J_n(u)}{\partial u_{:,k}}\Big|_{u^{\star}} \Big|\Big|^2 \leq \text{constant} \cdot \|\bar u - u^{\star}\|^2 = O(\epsilon)
\end{align}
The first equality holds because
\begin{subequations}
  \begin{align}
    \frac{\partial J(u)}{\partial u_{:,k}}\Big|_{\bar u} =& \frac{\partial }{\partial u_{:,k}}\Big|_{\bar x, \bar u} \sum_{i=0}^T c_{n,i}(x_i, u_{:,i}) = \frac{\partial }{\partial u_{:,k}}\Big|_{\bar x, \bar u} \sum_{i=k}^T c_{n,i}(x_i, u_{:,i}) \\
    =& \frac{\partial }{\partial u_{:,k}}\Big|_{\bar x, \bar u} c_{n,k}(x_k, u_{:,k}) + \frac{\partial }{\partial u_{:,k}}\Big|_{\bar x, \bar u} \sum_{i=k+1}^T c_{n,i}(x_i, u_{:,i}) \\
    =& M_{n,k}^{1u} + \frac{\partial }{\partial x_{k+1}} \Big|_{\bar x, \bar u} \sum_{i=k+1}^T c_{n,i}(x_i, u_{:,i}) \frac{\partial x_{k+1}}{\partial u_{:,k}} \Big|_{\bar x, \bar u} \\
    =& M^{1u}_{n,k} + \Omega_{n,k+1} B_k
  \end{align}
\end{subequations}
\hfill \QED

These results are used implicitly in the later proofs of lemmas.

\if\MODE\ARXIV{
\subsection{Proof of Lemma~\ref{lem:newton_game}} \label{app:newton_game}
First, we prove that the dynamics constraints
(\ref{eq:newton_dp_dynamics0}) and (\ref{eq:newton_dp_dynamics1})
are inductive definitions of the following approximation terms:
\begin{subequations}
  \label{eq:shorthands}
  \begin{align}
    \label{eq:newton_dp_states0}
    & \delta x_k = \sum_{i = 0}^{T} \frac{\partial x_k}{\partial u_{:,i}} \Big|_{\bar x, \bar{u}} \delta u_{:,i} \\
    \label{eq:newton_dp_states1}
    & \Delta x_k^l = \sum_{i = 0}^{T} \sum_{j = 0}^{T} \delta
      u_{:,i}^\top {\frac{\partial^2 x_k^l}{\partial u_{:,i} \partial
      u_{:,j}}} \Big|_{\bar x,\bar{u}} \delta u_{:,j}, \  l = 1, 2,
      \ldots, n_x
  \end{align}
\end{subequations}
Note that $x_0$ is fixed so that \eqref{eq:shorthands} holds at
$k=0$. Now we handle each of the terms inductively.

For $\delta x_{k+1}$, we have
  \begin{align}
    & \delta x_{k+1} = \sum_{i=0}^T \frac{\partial x_{k+1}}{\partial u_{:,i}} \Big|_{\bar x, \bar u} \delta u_{:,i} \nonumber \\
    &= \sum_{i=0}^T \frac{\partial f_k(x_k, u_{:,k})}{\partial u_{:,i}} \Big|_{\bar x, \bar u} \delta u_{:,i} \nonumber \\
    &= \frac{\partial f_k(x_k, u_{:,k})}{\partial x_k} \Big|_{\bar x, \bar u} \sum_{i=0}^T  \frac{\partial x_k} {\partial u_{:,k}} \delta u_{:,i} + \frac{\partial f_k(x_k, u_{:,k})}{\partial u_{:,k}} \Big|_{\bar x, \bar u} \delta u_{:,k} \nonumber \\
    & = A_k \delta x_k + B_k \delta u_{:,k}
  \end{align}
  We used the fact that $\frac{\partial f(x_k, u_{:,k})}{\partial
    u_{:,i}}$ is zero unless $i=k$.

  For $\Delta x_{k+1}$, row $l$ is given by:
  \begin{subequations}
    \begin{align}
      & \Delta x_{k+1}^l = \sum_{i = 0}^{T} \sum_{j = 0}^{T} \delta u_{:,i}^\top {\frac{\partial^2 f^l_k(x_k, u_{:,k})}{\partial u_{:,i} \partial u_{:,j}}} \Big|_{\bar{u}} \delta u_{:,j} \\
      & = \sum_{i = 0}^{T} \sum_{j = 0}^{T} \delta u_{:,i}^\top \left( \frac{\partial^2 f^l_k}{\partial u_{:,i} \partial u_{:,j}} + (\frac{\partial x_k}{\partial u_{:,i}})^\top \frac{\partial^2 f^l_k}{\partial x_k^2} \frac{\partial x_k}{\partial u_{:,j}} \right) \Big|_{\bar{u}} \delta u_{:,j} \nonumber \\
      & + \sum_{i = 0}^{T} \sum_{j = 0}^{T} \delta u_{:,i}^\top \left( (\frac{x_k}{\partial u_{:,i}})^\top \frac{\partial^2 f^l_k}{\partial x_k \partial u_{:,j}} + \frac{\partial^2 f^l_k}{\partial u_{:,i} \partial x_k} \frac{x_k}{\partial u_{:,j}} \right) \Big|_{\bar{u}} \delta u_{:,j} \nonumber \\
      & + \sum_{i = 0}^{T} \sum_{j = 0}^{T} \delta u_{:,i}^\top \left( \sum_{p = 1}^{n_x} \frac{\partial f^l_k}{\partial x^p_k} \frac{\partial^2 x^p_k}{\partial u_{:,i} \partial u_{:,j}} \right) \Big|_{\bar{u}} \delta u_{:,j} \\
      & = \delta u_{:,k}^\top \frac{\partial^2 f^l_k}{\partial u_{:,k}^2} \delta u_{:,k} + \delta x_k^\top \frac{\partial^2 f^l_k}{\partial x_k^2} \delta x_k + \delta x_k^\top \frac{\partial^2 f^l_k}{\partial x_k \partial u_{:,k}} \delta u_{:,k} \nonumber \\
      \label{eq:newton_dp_state1_der0}
      & + \delta u_{:,k}^\top \frac{\partial^2 f^l_k}{\partial u_{:,k} \partial x_k} \delta x_k + \sum_{p=1}^{n_x} \frac{\partial f^l_k}{\partial x^p_k} \sum_{i = 0}^{T} \sum_{j = 0}^{T} \delta u_{:,i}^\top {\frac{\partial^2 x^p_k}{\partial u_{:,i} \partial u_{:,j}}} \delta u_{:,j} \\
      \label{eq:newton_dp_state1_der1}
      & =
        \begin{bmatrix}
          \delta x_k \\
          \delta u_{:,k}
        \end{bmatrix}^\top
      G^l_k
      \begin{bmatrix}
        \delta x_k \\
        \delta u_{:,k}
      \end{bmatrix} + \sum_{p=1}^{n_x} A^{lp}_k \Delta x_k^p
    \end{align}
  \end{subequations}
  To get to each terms in (\ref{eq:newton_dp_state1_der0}), we used the fact that
  \begin{subequations}
    \begin{align}
      & \frac{\partial^2 f^l_k}{\partial u_{:,i} \partial u_{:,j}} = 0, \ \text{for}\  i \neq k \ \text{or}\  j \neq k \\
      & \delta x_k = \sum_{i = 0}^T \frac{\partial x_k}{\partial u_{:,i}} \delta u_{:,i} = \sum_{i = 0}^{k-1} \frac{\partial x_k}{\partial u_{:,i}} \delta u_{:,i}
    \end{align}
  \end{subequations}
  To get to (\ref{eq:newton_dp_state1_der1}), we used the fact
  \begin{subequations}
    \begin{align}
      & \frac{\partial f^l_k}{\partial x^p_k} = A^{lp}_k \\
      & \sum_{i = 0}^{T} \sum_{j = 0}^{T} \delta u_{:,i}^\top {\frac{\partial^2 x^p_k}{\partial u_{:,i} \partial u_{:,j}}} \Big|_{\bar{u}} \delta u_{:,j} = \Delta x_k^p \\
      &
        \begin{bmatrix}
          \frac{\partial^2 f^l_k}{\partial x_k^2} & \frac{\partial^2 f^l_k}{\partial x_k   \partial u_{:,k}} \\
          \frac{\partial^2 f^l_k}{\partial x_k \partial u_{:,k}} & \frac{\partial^2 f^l_k}{\partial u_{:,k}^2}
        \end{bmatrix} \Bigg|_{\bar x, \bar{u}} = G^l_k
    \end{align}
  \end{subequations}
  Both $l$ and $p$ are used to pick out the corresponding element for a vector or matrix. $A^{lp}_k$ means the $l$th row and $p$th column of matrix $A_k$. Equation (\ref{eq:newton_dp_state1_der1}) actually describes each element in (\ref{eq:newton_dp_states1}), so we've proven that both are true.

  Next we prove (\ref{eq:newton_dp_objective}) is the quadratic approximation of $J_n(u)$, i.e.
  \begin{align}
    \label{eq:newton_dp_obj_equiv}
    & \text{quad}(J_n(u))_{\bar u} =
      J_n(\bar u) + \frac{\partial J_n^T(\bar u)}{\partial u} + \frac{1}{2} \delta u^T \frac{\partial^2 J_n(\bar u)}{\partial u^2} \delta u \nonumber \\
    &= \frac{1}{2} \sum_{k=0}^T \left(
      \begin{bmatrix}
        1 \\
        \delta x_k \\
        \delta u_{:,k}
      \end{bmatrix}^T
    M_{n, k}
    \begin{bmatrix}
      1 \\
      \delta x_k \\
      \delta u_{:,k}
    \end{bmatrix}
    + M^{1k}_{n,k} \Delta x_k \right)
  \end{align} \\

  We'll need the explicit expressions for the associated derivatives.
  \begin{subequations}
    \label{eq:Jderivatives}
    \begin{align}
      & \frac{\partial J_n(u)}{\partial u_{:,i}} = \sum_{k=0}^T \left( \frac{\partial c_{n, k}(x_k,u_{:,k})}{\partial u_{:,i}} + \frac{\partial c_{n, k}(x_k,u_{:,k})}{\partial x_k} \frac{\partial x_k}{\partial u_{:,i}} \right) \\
      & \frac{\partial^2 J_n(u)}{\partial u_{:,i} \partial u_{:,j}} \nonumber \\
      & = \sum_{k=0}^T \left( \frac{\partial^2 c_{n, k}(x_k,u_{:,k})}{\partial u_{:,i} \partial u_{:,j}} + \frac{\partial x_k}{\partial u_{:,i}}^\top \frac{\partial^2 c_{n, k}(x_k,u_{:,k})}{\partial x_k^2} \frac{\partial x_k}{\partial u_{:,j}} \right) \nonumber \\
      & + \sum_{k=0}^T \left( \frac{\partial x_k}{\partial u_{:,i}}^\top \frac{\partial^2 c_{n, k}(x_k,u_{:,k})}{\partial x_k \partial u_{:,j}} + \frac{\partial^2 c_{n, k}(x_k,u_{:,k})}{\partial u_{:,i} \partial x_k} \frac{\partial x_k}{\partial u_{:,j}} \right) \nonumber \\
      \label{eq:smallHessian}
      & +\sum_{k=0}^T \sum_{l=1}^{n_x} \frac{\partial c_{n, k}(x_k,u_{:,k})}{\partial x_k^l} \frac{\partial^2 x_k^l}{\partial u_{:,i} \partial u_{:,i}}
    \end{align}
  \end{subequations}

  We break down each term in (\ref{eq:newton_dp_obj_equiv}). First the second order term.
  \begin{subequations}
    \label{eq:newton_dp_objective_second}
    \begin{align}
      & \delta u^T \frac{\partial^2 J_n(\bar u)}{\partial u^2} \delta u =
        \sum_{i,j=0}^T \delta u_{:,i}^T \frac{\partial^2 J_n(u)}{\partial u_{:,i} \partial u_{:,j}} \bigg|_{\bar u} \delta u_{:,j} \\
      & =  \sum_{i,j,k=0}^T \delta u_{:,i}^T \left(
        \frac{\partial^2 c_{n,k}(x_k,u_{:,k})}{\partial u_{:,i} \partial u_{:,j}} \bigg|_{\bar u} \right) \delta u_{:,j} \nonumber \\
      & + \sum_{i,j,k=0}^T \delta u_{:,i}^T \left( \frac{\partial x_k}{\partial u_{:,i}}^\top
        \frac{\partial^2 c_{n,k}(x_k,u_{:,k})}{\partial x_k^2} \bigg|_{\bar u}
        \frac{\partial x_k}{\partial u_{:,j}}
        \right) \delta u_{:,j} \nonumber \\
      & + \sum_{i,j,k=0}^T \delta u_{:,i}^T \left(
        \frac{\partial x_k}{\partial u_{:,i}}^\top
        \frac{\partial^2 c_{n,k}(x_k,u_{:,k})}{\partial x_k \partial u_{:,j}} \bigg|_{\bar u} \right) \delta u_{:,j} \nonumber \\
      \label{eq:newton_dp_objective_sec0}
      & + \sum_{i,j,k=0}^T \delta u_{:,i}^T \left( \frac{\partial^2 c_{n,k}(x_k,u_{:,k})}{\partial u_{:,i} \partial x_k} \bigg|_{\bar u}
        \frac{\partial x_k}{\partial u_{:,j}} \right) \delta u_{:,j} \nonumber \\
      & + \sum_{k=0}^T \sum_{p=1}^n
        \frac{\partial c_{n,k}(x_k,u_{:,k})}{\partial x_k^p} \bigg|_{\bar u}
        \sum_{i,j=0}^T \delta u_{:,i}^T \frac{\partial^2 x_k^p}{\partial u_{:,i} \partial u_{:,i}} \delta u_{:,j} \\
      \label{eq:newton_dp_objective_sec2}
      = & \sum_{k=0}^T \left( \delta u_{:,k}^\top \frac{\partial^2 c_{n,k}}{\partial u_{:,k}^2 } \bigg|_{\bar u} \delta u_{:,k} + \delta x_k^\top \frac{\partial^2 c_{n,k}}{\partial x_k^2} \bigg|_{\bar u} \delta x_k \right) + \nonumber \\
      & \sum_{k=0}^T  \left( \delta x_k ^\top \frac{\partial^2 c_{n,k}}{\partial x_k \partial u_{:,k}}\vert_{\bar u}\delta u_{:,k} + \delta u_{:,k} ^\top \frac{\partial^2 c_{n,k}}{\partial u_{:,k} \partial x_k} \bigg|_{\bar u} \delta x_k\right) + \nonumber \\
      & \sum_{k=0}^T \sum_{p=1}^n
        \frac{\partial c_{n,k}(x_k,u_{:,k})}{\partial x_k^p} \Delta x_k^p \\
      = & \sum_{k=0}^T \left(
          \begin{bmatrix}
            \delta x_k \\
            \delta u_{:,k}
          \end{bmatrix}^T
      \begin{bmatrix}
        \frac{\partial^2 c_{n,k}}{\partial x_k^2}
        & \frac{\partial^2 c_{n,k}}{\partial x_k \partial u_{:,k}} \\
        \frac{\partial^2 c_{n,k}}{\partial u_{:,k}\partial x_k}
        & \frac{\partial^2 c_{n,k}}{\partial u_{:,k}^2}
      \end{bmatrix} \bigg|_{\bar u}
          \begin{bmatrix}
            \delta x_k \\
            \delta u_{:,k}
          \end{bmatrix} \right) \nonumber \\
      & + \left( \frac{\partial c_{n,k}(x_k,u_{:,k})}{\partial x_k} \Delta x_k \right)
    \end{align}
  \end{subequations}
  The first term in \eqref{eq:newton_dp_objective_sec0} to \eqref{eq:newton_dp_objective_sec2} holds because $c_{n,k}(x_k,u_{:,k})$ only depends directly on $u_{:,i}$ and $u_{:,j}$ when $i=j=k$. The others hold because  $x_k$ only depends on $u_{:,i}$ and $u_{:,j}$ when $i,j <k$. The last term uses the definition of $\Delta x_k$ in \eqref{eq:newton_dp_states1}.

  The first order term
  \begin{subequations}
    \label{eq:newton_dp_objective_first}
    \begin{align}
      & \frac{\partial J_n(\bar u)}{\partial u} = \sum_{i=0}^T \frac{\partial J_n(u)}{\partial u_{:,i}} \delta u_{:,i} \\
      & = \sum_{i,k=0}^T \left( \frac{\partial c_{n,k}(x_k,u_{:,k})}{\partial u_{:,i}} + \frac{\partial c_{n,k}(x_k,u_{:,k})}{\partial x_k} \frac{\partial x_k}{\partial u_{:,i}} \right) \delta u_{:,i} \\
      & = \sum_{k=0}^T \left( \frac{\partial c_{n,k}(x_k,u_{:,k})}{\partial u_{:,k}} \delta u_{:,k} + \frac{\partial c_{n,k}(x_k,u_{:,k})}{\partial x_k} \sum_{i=0}^T \frac{\partial x_k}{\partial u_{:,i}} \right) \\
      & = \sum_{k=0}^T \left( \frac{\partial c_{n,k}(x_k,u_{:,k})}{\partial u_{:,k}} \delta u_{:,k} + \frac{\partial c_{n,k}(x_k,u_{:,k})}{\partial x_k} \delta x_k\right)
    \end{align}
  \end{subequations}
  And constant term
  \begin{align}
    \label{eq:newton_dp_objective_constant}
    J_n(\bar u) = \sum_{k=0}^T c_{n,k}(\bar x_k, \bar u_{:,k})
  \end{align}
  From (\ref{eq:newton_dp_objective_second}),
  (\ref{eq:newton_dp_objective_first}), and
  (\ref{eq:newton_dp_objective_constant}) it follows that (\ref{eq:newton_dp_obj_equiv}) is true.
\hfill\QED
}
\fi

\subsection{Proof of Lemma~\ref{lem:matrixClose}}
\label{app:matrixClose}
A more complete version of Lemma \ref{lem:matrixClose} is
\begin{subequations}
  \label{eq:closeness}
  \begin{align}
    \label{eq:close_D}
    \tilde D_{n,k} &= D_{n,k} + O(\epsilon) \\
    \label{eq:close_S0}
    S^{1x}_{n,k} &= \Omega_{n,k} + O(\epsilon) \\
    \label{eq:close_S1}
    \tilde S^{1x}_{n,k} &= \Omega_{n,k} + O(\epsilon) \\
    \label{eq:close_S2}
    \tilde S^{1x}_{n,k} &= S^{1x}_{n,k} + O(\epsilon^2) \\
    \label{eq:close_S3}
    \tilde S_{n,k} &= S_{n,k} + O(\epsilon) \\
    \label{eq:close_gamma0}
    \begin{bmatrix}
      \tilde \Gamma^{1x}_{n,k} & \tilde \Gamma^{1u}_{n,k}
    \end{bmatrix} &=
                    \begin{bmatrix}
                      \Gamma^{1x}_{n,k} & \Gamma^{1u}_{n,k}
                    \end{bmatrix} + O(\epsilon^2) \\
    \label{eq:close_gamma1}
    \tilde \Gamma_{n,k} &= \Gamma_{n,k} + O(\epsilon)\\
    \label{eq:close_gamma2}
    \Gamma_{n,k}^{1u} &= O(\epsilon) \\
    \label{eq:close_gamma3}
    \tilde \Gamma_{n,k}^{1u} &= O(\epsilon) \\
    \label{eq:close_F}
    \tilde F_k &= F_k + O(\epsilon) \\
    \label{eq:close_P}
    \tilde P_k &= P_k + O(\epsilon) \\
    \label{eq:close_H0}
    \tilde H_k &= H_k + O(\epsilon^2) \\
    \label{eq:close_H1}
    \tilde H_k &= O(\epsilon) \\
    \label{eq:close_H2}
    H_k &= O(\epsilon) \\
    \label{eq:close_s0}
    \tilde s_k &= s_k + O(\epsilon^2) \\
    \label{eq:close_s1}
    \tilde s_k &= O(\epsilon) \\
    \label{eq:close_s2}
    s_k &= O(\epsilon) \\
    \label{eq:close_K}
    \tilde K_k &= K_k + O(\epsilon),
  \end{align}
\end{subequations}
of which we give a proof by induction for in this section. \\

For $k = T$, because of the way these variables are constructed, they are identical, i.e.
\begin{subequations}
  \begin{align}
    \Gamma_{n,T} =& \tilde \Gamma_{n,T} \\
    F_T =& \tilde F_T \\
    P_T =& \tilde P_T \\
    H_T =& \tilde H_T \\
    s_T =& \tilde s_T \\
    K_T =& \tilde K_T \\
    S_{n,T} =& \tilde S_{n,T}
  \end{align}
\end{subequations}
So we have  (\ref{eq:close_S2})(\ref{eq:close_S3})(\ref{eq:close_gamma0})(\ref{eq:close_gamma1})(\ref{eq:close_F})(\ref{eq:close_P})(\ref{eq:close_H0})(\ref{eq:close_s0})(\ref{eq:close_K}) hold for $k = T$.

We also know that
\begin{align}
  M_{n,T}^{1u} = \frac{\partial c_{n,T}}{\partial u_T} = \frac{\partial J_n(u)}{\partial u_T}=O(\epsilon)
\end{align} where the first equality is by construction, the second is true because $u_T$ only appears in $J_n(u)$ in $c_{n,T}$. By construction, $\Gamma_{n,T}^{1u} = \tilde \Gamma_{n,T}^{1u} = M_{n,T}^{1u} = O(\epsilon)$, so (\ref{eq:close_gamma2})(\ref{eq:close_gamma3}) are true for $k = T$. \\
Similarly, $H_T$ and $\tilde H_T$ are constructed from $\Gamma_{n,T}^{u1}$ and $\tilde \Gamma_{n,T}^{u1}$, so (\ref{eq:close_H1})(\ref{eq:close_H2}) are true for $k=T$.

Because $F_k^{-1}$ is bounded above, $s_T = - F_T^{-1} H_T = - F_T^{-1} O(\epsilon) = O(\epsilon)$. Similarly, $\tilde s_T = O(\epsilon)$. Equations (\ref{eq:close_s1})(\ref{eq:close_s2}) are true for $k=T$.

From (\ref{eq:newton_dp_solution_S}) and $\Gamma_{n,T} = M_{n,T}$, we can get
\begin{subequations}
  \begin{align}
    S_{n,T}^{1x} =& M_{n,T}^{1x} + M_{n,T}^{1u} K_T + s_T^\top ( M_{n,T}^{ux} + M_{n,T}^{uu} K_T ) \\
    =& \Omega_{n,T} + O(\epsilon) K_T + O(\epsilon) ( M_{n,T}^{ux} + M_{n,T}^{uu} K_T ) \\
    =& \Omega_{n,T} + O(\epsilon)
  \end{align}
\end{subequations}
because $M_{n,T}$ is bounded. Hence (\ref{eq:close_S0}) is true. Further, (\ref{eq:close_S1}) is also true.

The time indices for $D_{n,T-1}$ and $\tilde D_{n,T-1}$ go to a maximum of $T-1$, so to prove things inductively, we need (\ref{eq:close_D}) to hold for $k=T-1$.The difference between constructions of $D_{n,T-1}$ and $\tilde D_{n,T-1}$ is in that the former uses $\Omega_{n,T}$ and the later uses $\tilde S_{n,T}^{1x}$. But since we've proven $\Omega_{n,T} = \tilde S_{n,T}^{1x} + O(\epsilon)$, and $G_{T-1}$ is bounded, we can also conclude $D_{n,T-1} = \tilde D_{n,T-1} + O(\epsilon)$. Therefore (\ref{eq:close_D}) is true for $k = T-1$.

So far, we've proved that for the last step, either $k = T$ or $k= T -1$, (\ref{eq:closeness}) is true. Assuming except for (\ref{eq:close_D}), (\ref{eq:closeness}) is true for $k+1$ and (\ref{eq:close_D}) is true for $k$. If we can prove all equations hold one step back, our proof by induction would be done.

Assume (\ref{eq:close_D}) holds for $k$ and other equations in (\ref{eq:closeness}) hold for $k+1$. Readers be aware that we'll use these assumptions implicitly in the derivations following.

From (\ref{eq:GammaDef}) we can get
\begin{subequations}
  \begin{align}
    \Gamma_{n,k}^{1u} =& M_{k,T}^{1u} + S_{n,k+1}^{1x} B_k \\
    =& M_{k,T}^{1u} + \Omega_{n,k+1} B_k + O(\epsilon) B_k \\
    =& O(\epsilon)
  \end{align}
\end{subequations}
Here we used (\ref{eq:J_grad_small}). Similarly, we can prove $\tilde \Gamma_{n,k}^{1u} = O(\epsilon)$. So (\ref{eq:close_gamma2}) and (\ref{eq:close_gamma3}) hold for $k$.

From (\ref{eq:GammaDef}) and (\ref{eq:TildeGammaBackprop}) we can compute the difference between $\tilde \Gamma_{n,k}$ and $\Gamma_{n,k}$ as
\begin{subequations}
  \begin{align}
    & \ \tilde \Gamma_{n,k} - \Gamma_{n,k} \nonumber \\
    = &
        \begin{bmatrix}
          \tilde S_{n, k+1}^{11} - S_{n, k+1}^{11} & \mathbf{0} & \mathbf{0} \\
          A_k^\top (\tilde S_{n,k+1}^{x1} - S_{n,k+1}^{x1}) & \mathbf{0} & \mathbf{0} \\
          B_k^\top  (\tilde S_{n,k+1}^{x1} - S_{n,k+1}^{x1}) & \mathbf{0} & \mathbf{0}
        \end{bmatrix} \nonumber \\
    & +
      \begin{bmatrix}
        \mathbf{0} & (\tilde S_{n,k+1}^{1x} - S_{n,k+1}^{1x})A_k & \mathbf{0} \\
        \mathbf{0} & A_k^\top (\tilde S_{n,k+1}^{xx} - S_{n,k+1}^{xx})A_k + (\tilde D_k^{xx} - D_k^{xx}) & \mathbf{0} \\
        \mathbf{0} & B_k^\top (\tilde S_{n,k+1}^{xx} - S_{n,k+1}^{xx}) A_k + (\tilde D_k^{ux} - D_k^{ux}) & \mathbf{0}
      \end{bmatrix} \nonumber \\
    & +
      \begin{bmatrix}
        \mathbf{0} & \mathbf{0} & (\tilde S_{n,k+1}^{1x} - S_{n,k+1}^{1x})B_k \\
        \mathbf{0} & \mathbf{0} & A_k^\top (\tilde S_{n,k+1}^{xx} - S_{n,k+1}^{xx}) B_k + (\tilde D_k^{xu} - D_k^{xu}) \\
        \mathbf{0} & \mathbf{0} & B_k^\top (\tilde S_{n,k+1}^{xx} - S_{n,k+1}^{xx}) B_k + (\tilde D_k^{uu} - D_k^{uu})
      \end{bmatrix} \\
    =& \begin{bmatrix}
      O(\epsilon) & A_k O(\epsilon^2) & B_k O(\epsilon^2) \\
      A_k^\top O(\epsilon^2) & A_k^\top O(\epsilon) A_k  + O(\epsilon) & A_k^\top O(\epsilon) B_k  + O(\epsilon) \\
      B_k^\top O(\epsilon^2)  & B_k^\top O(\epsilon) A_k + O(\epsilon) & B_k^\top O(\epsilon) B_k + O(\epsilon)
    \end{bmatrix} \\
    =&
       \begin{bmatrix}
         O(\epsilon) & O(\epsilon^2) & O(\epsilon^2) \\
         O(\epsilon^2) & O(\epsilon) & O(\epsilon) \\
         O(\epsilon^2) & O(\epsilon) & O(\epsilon)
       \end{bmatrix}
  \end{align}
\end{subequations}
from which we can see that (\ref{eq:close_gamma0}) and (\ref{eq:close_gamma1}) are true.
Once we proved the closeness between $\tilde \Gamma_{n,k}$ and $\Gamma_{n,k}$ and the specific terms are $O(\epsilon)$, i.e. (\ref{eq:close_gamma0}) to (\ref{eq:close_gamma3}), because of they way they are constructed from $\tilde \Gamma_{n,k}$ and $\Gamma_{n,k}$, we are safe to say
\begin{subequations}
  \begin{align}
    \tilde F_k &= F_k + O(\epsilon) \\
    \tilde P_k &= P_k + O(\epsilon) \\
    \tilde H_k &= H_k + O(\epsilon^2) \\
    \tilde H_k &= O(\epsilon) \\
    H_k &= O(\epsilon)
  \end{align}
\end{subequations}
Therefore, (\ref{eq:close_F}), (\ref{eq:close_P}), (\ref{eq:close_H0}), (\ref{eq:close_H1}) and (\ref{eq:close_H2}) are true for $k$.

Now that we have the results with $F_k$, $\tilde F_k$, $H_k$, $\tilde H_k$, $P_k$ and $\tilde P_k$, we can move to what are immediately following, i.e. $s_k$, $\tilde s_k$, $K_k$ and $\tilde K_k$.
\begin{align}
  s_k =& - F_k^{-1} H_k = - F_k^{-1} O(\epsilon) = O(\epsilon)
\end{align}
which is true because $F_k^{-1}$ is bounded above. Similarly, we have $\tilde s_k = O(\epsilon)$. Equations (\ref{eq:close_s1}) and (\ref{eq:close_s2}) are true.

\begin{subequations}
  \begin{align}
    \tilde s_k &= - \tilde F_k^{-1} \tilde H_k = - (F_k + O(\epsilon))^{-1} (H_k + O(\epsilon^2)) \\
               &= - (F_k^{-1} + O(\epsilon)) (H_k + O(\epsilon^2)) \\
               &= - F_k^{-1} H_k + F_k^{-1} O(\epsilon^2) + H_k O(\epsilon) + O(\epsilon^2) \\
               &= s_k + O(\epsilon^2) \\
    \tilde K_k &= - \tilde F_k^{-1} \tilde P_k = - (F_k + O(\epsilon))^{-1} (P_k + O(\epsilon)) \\
               &= - (F_k^{-1} + O(\epsilon)) (P_k + O(\epsilon)) \\
               &= - F_k^{-1} P_k + (F_k^{-1} + P_k) O(\epsilon) + O(\epsilon^2) \\
               &= K_k + O(\epsilon)
  \end{align}
\end{subequations}
Equations (\ref{eq:close_s0}) and (\ref{eq:close_K}) are true for $k$.

Now we are equipped to get closeness/small results for $S_{n,k}$ and $\tilde S_{n,k}$.
\begin{subequations}
  \begin{align}
    & \ \tilde S_{n,k} - S_{n,k} \nonumber \\
    =&
       \begin{bmatrix}
         1 & 0 & 0 \\
         0 & I & \tilde K_k^\top
       \end{bmatrix} \tilde \Gamma_{n, k}
                 \begin{bmatrix}
                   1 & 0 \\
                   0 & I \\
                   0 & \tilde K_k
                 \end{bmatrix} -
                       \begin{bmatrix}
                         1 & 0 & 0 \\
                         0 & I & K_k^\top
                       \end{bmatrix} \Gamma_{n, k}
                                 \begin{bmatrix}
                                   1 & 0 \\
                                   0 & I \\
                                   0 & K_k
                                 \end{bmatrix} \nonumber \\
    & +
      \begin{bmatrix}
        \tilde s_k^\top \tilde \Gamma_{n, k}^{uu} \tilde s_k + 2 \tilde s_k^\top \tilde \Gamma_{n, k}^{u1} - s_k^\top \Gamma_{n, k}^{uu} s_k - 2 s_k^\top \Gamma_{n, k}^{u1} & \mathbf{0} \\
        (\tilde \Gamma_{n, k}^{xu} + \tilde \Gamma_{n, k}^{uu} \tilde K_k) \tilde s_k - (\Gamma_{n, k}^{ux} + \Gamma_{n, k}^{uu} K_k)^\top s_k & \mathbf{0}
      \end{bmatrix} \nonumber \\
    & +
      \begin{bmatrix}
        \mathbf{0} & \tilde s_k^\top (\tilde \Gamma_{n, k}^{ux} + \tilde \Gamma_{n, k}^{uu} \tilde K_k) - s_k^\top (\Gamma_{n, k}^{ux} + \Gamma_{n, k}^{uu} K_k) \\
        \mathbf{0} & \mathbf{0}
      \end{bmatrix} \\
    =&
       \begin{bmatrix}
         \tilde \Gamma_{n,k}^{11} & \tilde \Gamma_{n,k}^{1x} + \tilde \Gamma_{n,k}^{1u} \tilde K \\
         \tilde \Gamma_{n,k}^{x1} + \tilde K^\top \tilde \Gamma_{n,k}^{u1} & \tilde \Gamma_{n,k}^{xx} + 2 \tilde \Gamma_{n,k}^{xu} \tilde K + \tilde K^\top \tilde \Gamma_{n,k}^{uu} \tilde K
       \end{bmatrix} \nonumber \\
    & -
      \begin{bmatrix}
        \Gamma_{n,k}^{11} & \Gamma_{n,k}^{1x} + \Gamma_{n,k}^{1u} K \\
        \Gamma_{n,k}^{x1} + K^\top \Gamma_{n,k}^{u1} & \Gamma_{n,k}^{xx} + 2 \Gamma_{n,k}^{xu} K + K^\top \Gamma_{n,k}^{uu} K
      \end{bmatrix} \nonumber \\
    & +
      \begin{bmatrix}
        O(\epsilon) & O(\epsilon^2) \\
        O(\epsilon) & 0
      \end{bmatrix} \\
    =&
       \begin{bmatrix}
         O(\epsilon) & O(\epsilon^2) \\
         O(\epsilon^2) & O(\epsilon)
       \end{bmatrix} \\
  \end{align}
\end{subequations}
So that (\ref{eq:close_S2}) and (\ref{eq:close_S3}) are true for $k$.

\begin{subequations}
  \begin{align}
    & S_{n,k}^{1x} = M_{n,k}^{1x} + S_{n,k+1}^{1x} A_k + \Gamma_{n,k}^{1u} K_k + s_k^\top ( \Gamma_{n,k}^{ux} + \Gamma_{n,k}^{uu} K_k ) \\
    & = M_{n,k}^{1x} + \Omega_{n,k+1} A_k + A_k O(\epsilon) \nonumber \\
    & \quad + K_k O(\epsilon) +  (\Gamma_{n,k}^{ux} + \Gamma_{n,k}^{uu} K_k) O(\epsilon) \\
    & = \Omega_{n,k} + O(\epsilon)
  \end{align}
\end{subequations}
Therefore, (\ref{eq:close_S0}) holds and then naturally (\ref{eq:close_S1}) holds.

We continue to prove that $\tilde D_{n,k-1}$ and $D_{n, k-1}$ are close, which is true because
\begin{subequations}
  \begin{align}
    \tilde D_{n,k-1} =& \sum_{l=1}^{n_x} \tilde S_{n,k}^{1x^l} G_k^l \\
    =& \sum_{l=1}^{n_x} (\Omega_{n,k}^l + O(\epsilon)) G_k^l \\
    =& \sum_{l=1}^{n_x} \Omega_{n,k}^l G_k^l + O(\epsilon) \\
    =& D_{n,k-1} + O(\epsilon)
  \end{align}
\end{subequations}
So (\ref{eq:close_D}) is true.
\hfill\QED

\subsection{Proof of Lemma~\ref{lem:solClose}} \label{app:solClose}
A more complete version of Lemma \ref{lem:solClose} is
\begin{subequations}
  \label{eq:closeness_forward}
  \begin{align}
    \label{eq:DDP_approx_states}
    \delta x_{k+1}^D =& A_k \delta x_k^D + B_k \delta u_k^D + O(\epsilon^2) \\
    \label{eq:newton_approx_states}
    \delta x_{k+1}^N =& A_k \delta x_k^N + B_k \delta u_k^N + O(\epsilon^2) \\
    \label{eq:input_small_newton}
    \delta u_k^N = O(\epsilon) \\
    \label{eq:input_small_ddp}
    \delta u_k^D = O(\epsilon) \\
    \label{eq:state_small_newton}
    \delta x_k^N = O(\epsilon) \\
    \label{eq:state_small_ddp}
    \delta x_k^D = O(\epsilon) \\
    \label{eq:input_step_close}
    \delta u_k^N - \delta u_k^D =& O(\epsilon^2)\\
    \label{eq:state_step_close}
    \delta x_k^N - \delta x_k^D =& O(\epsilon^2)\\
    \label{eq:close_update_newton_DDP}
    \delta u^N - \delta u^D =& O(\epsilon^2) \\
    \label{eq:classic_newton_convergence}
    \|\bar u + \delta u^N - u^{\star}\| =& O(\epsilon^2). \\
    \label{eq:DDP_convergence}
    \|\bar u + \delta u^D - u^{\star}\| =& O(\epsilon^2).
  \end{align}
\end{subequations}

Equation (\ref{eq:DDP_approx_states}) comes directly from the Taylor series expansion of (\ref{eq:dynamics}) and (\ref{eq:newton_approx_states}) from (\ref{eq:newton_dp_dynamics0}).

We prove (\ref{eq:input_small_newton}) to (\ref{eq:state_step_close}) by induction. For $k=0$, $\delta x_{0}^N = \delta x_{0}^D = 0$ and $\delta u_{0}^N = s_0,\  \delta x_{0}^D = \tilde s_0$. We know from the proof of lemma \ref{lem:matrixClose} that $s_0 = \tilde s_0 + O(\epsilon^2)$, $s_0 = O(\epsilon)$ and $\tilde s_0 = O(\epsilon)$,  so (\ref{eq:input_small_newton}) to (\ref{eq:state_step_close}) hold for $k=0$. Assume (\ref{eq:input_small_newton}) to (\ref{eq:state_step_close}) hold for $k$, then
\begin{subequations}
  \begin{align}
    \delta u^N_{k+1} &= K_{k+1} \delta x^N_k + s_{k+1} = O(\epsilon) \\
    \delta u^D_{k+1} &= \tilde K_{k+1} \delta x^N_k + \tilde s_{k+1} = O(\epsilon) \\
    \delta x^N_{k+1} =& A_k \delta x^N_k + B_k \delta u^N_k + O(\epsilon^2) = O(\epsilon) \\
    \delta x^D_{k+1} =& A_k \delta x^D_k + B_k \delta u^D_k + O(\epsilon^2) = O(\epsilon) \\
    \delta u^N_{k+1} - \delta u^D_{k+1} =& K_k \delta x^N_{k} - \tilde K_k \delta x^D_{k} + s_k - \tilde s_k \\
    =& K_k \delta x^N_k - (K_k + O(\epsilon)) (\delta x^N_k + O(\epsilon^2)) \nonumber \\
                     & + O(\epsilon^2) \\
    =& O(\epsilon) \delta x^N_k + O(\epsilon^2) \\
    =& O(\epsilon^2) \\
    \delta x^N_{k+1} - \delta x^D_{k+1} =& A_k ( \delta x^N_{k} - \delta x^D_{k} ) + B_k ( \delta u^N_{k} - \delta u^D_{k} ) + O(\epsilon^2) \nonumber \\
    =& O(\epsilon^2)
  \end{align}
\end{subequations}\\
So (\ref{eq:input_small_newton}) to (\ref{eq:state_step_close}) hold for $k+1$ and the proof by induction is done. Equation (\ref{eq:close_update_newton_DDP}) comes directly as a result. Equation (\ref{eq:classic_newton_convergence}) is classic convergence analysis for Newton's method \cite{nocedal2006numerical}. Equation (\ref{eq:DDP_convergence}) follows directly from (\ref{eq:close_update_newton_DDP}) and (\ref{eq:classic_newton_convergence}).
\hfill\QED

\subsection{Proof of Lemma~\ref{lem:quadSol}}
\label{app:quadSol}
As discussed in the proof of Lemma~\ref{lem:ddp_matrices}, a necessary
condition for the solution of \eqref{eq:quadQGame} is given by
\eqref{eq:ddp_necessary}. Thus, a sufficient condition for a unique
solution is that $\tilde F_k$ be invertible. At the beginning of the
appendix, whe showed that $F_k^{-1}$ exists near $u^\star$ and that
its spectral radius is bounded. 

Now we show that $\tilde F_k$ exists and $\rho(\tilde F_k^{-1})$ is bounded. 
Lemma~\ref{lem:matrixClose} implies that $\tilde F_k = F_k +
O(\epsilon)$. It follows that
\begin{equation*}
  \tilde F_k^{-1} = (F_k + O(\epsilon))^{-1} = F_{k}^{-1} - F_k^{-1}
  O(\epsilon) F_k^{-1} = F_k^{-1} + O(\epsilon).
\end{equation*}
It follows that $\tilde F_k^{-1}$ exists and is bounded in a
neighborhood of $u^\star$.
\hfill\QED

\addtolength{\textheight}{-12cm}   

\end{document}